\documentclass[a4paper,11pt]{article}
\usepackage[utf8]{inputenc}
\usepackage[T1]{fontenc}

\usepackage{amssymb}
\usepackage{amsmath}
\usepackage{amsthm}
\usepackage{amsopn}
\usepackage{graphicx}
\usepackage{mathrsfs}
\usepackage{empheq}
\usepackage{bm}
\usepackage{relsize}
\usepackage{breakcites}
\usepackage{nicefrac}

\usepackage{float}
\usepackage{amsfonts}
\usepackage{dsfont}
\usepackage{url}
\usepackage{pifont}
\usepackage[normalem]{ulem}
\usepackage{algorithm}
\usepackage{algorithmic}
\usepackage{caption}
\usepackage{xspace}
\usepackage{enumerate}
\usepackage{bbm}
\usepackage[tracking=true]{microtype}

\usepackage{tikz}
\usetikzlibrary{fit,positioning}
\usetikzlibrary{arrows}

\usepackage{natbib}
\bibpunct{(}{)}{;}{a}{,}{,}
\usepackage{hyperref}
\hypersetup{
  colorlinks = true,
  urlcolor = blue,
  linkcolor = blue,
  citecolor = blue,
}

\usepackage{cleveref}
\usepackage{autonum}

\usepackage{geometry}
\let\bb\mathbb       

  \def\Ex{{\bb E}}

  \def\BB{{\bb B}}
  \def\RR{{\bb R}}

\def\Ex{\mathbf E}

\def\calC{\mathcal C}
\def\calD{\mathcal D}

\def\calF{\mathcal F}
\def\calG{\mathcal G}
\def\calH{\mathcal H}
\def\calR{\mathcal R}
\def\calX{\mathcal X}

\def\bb{\mathbb}
\def\hat{\widehat}
\def\bfX{\mathbf X}

\def\bfA{\mathbf A}

\def\bSigma{\boldsymbol\Sigma}

\def\bTau{\boldsymbol{\mathcal T}}

\def\bPi{\boldsymbol\Pi}
\def\bX{\boldsymbol X}

\def\bx{\boldsymbol x}

\def\bv{\boldsymbol v}

\def\bbeta{\boldsymbol \beta}

\def\bpi{\boldsymbol \pi}
\def\bomega{\boldsymbol \omega}

\def\bZ{\boldsymbol Z}
\def\11{\mathds 1}
\def\b1{\mathbf 1}
\def\Pb{\mathbf P}
\def\var{{\bf Var}}

\def\KL{{\rm KL}}
\def\bfZ{\mathbf Z}
\def\bfI{\mathbf I}

\DeclareMathOperator*{\argmin}{arg\,min}

\newtheorem{proposition}{Proposition}

\newtheorem{lemma}{Lemma}[section]
\newtheorem{prop}{Proposition}[section]
\newtheorem{theorem}{Theorem}[section]

\usepackage{lineno,xcolor}

\title{Optimal Kullback-Leibler Aggregation in Mixture Density Estimation by Maximum Likelihood}
\author{Arnak S. Dalalyan and Mehdi Sebbar}

\begin{document}
\maketitle

\begin{abstract}
We study the maximum likelihood estimator of density of $n$ independent observations,
under the assumption that it is well approximated by a mixture with a large number of
components. The main focus is on statistical properties with respect to the
Kullback-Leibler loss. We establish risk bounds taking the form of sharp oracle inequalities
both in deviation and in expectation. A simple consequence of these bounds is that the
maximum likelihood estimator attains the optimal rate $((\log K)/n)^{\nicefrac12}$, up to a
possible logarithmic correction, in the problem of convex aggregation when the number
$K$ of components is larger than $n^{\nicefrac12}$. More importantly, under the additional
assumption that the Gram matrix of the components satisfies the compatibility condition,
the obtained oracle inequalities yield the optimal rate in the sparsity scenario. That
is, if the weight vector is (nearly) $D$-sparse, we get the rate $(D\log K)/n$. As a
natural complement to our oracle inequalities, we introduce the notion of nearly-$D$-sparse
aggregation and establish matching lower bounds for this type of aggregation.
\end{abstract}

\section{Introduction}
Assume that we observe $n$ independent random vectors $\bX_1,\ldots,\bX_n\in\calX$ drawn from a probability distribution $P^*$
that admits a density function $f^*$ with respect to some reference measure $\nu$.
The goal is to estimate the unknown density by a mixture density. More precisely, we assume that for a given family of mixture
components $f_1,\ldots,f_K$, the unknown density of the observations $f^*$ is well approximated by a convex combination
$f_{\bpi}$ of these components, where
\begin{equation}
f_{\bpi}(\bx)=\sum_{j=1}^K\pi_j f_j(\bx), \quad \bpi \in \BB_+^K=\Big\{\bpi\in [0,1]^K: \sum_{j=1}^K\pi_j=1\Big\}.
\end{equation}
The assumption that the component densities $\calF=\{f_j:j\in[K]\}$ are known essentially means that they are chosen from a
dictionary obtained on the basis of previous experiments or expert knowledge.

We focus on the problem of estimation of the density function $f_{\bpi}$ and the weight vector $\bpi$ from the simplex
$\BB_+^K$ under the sparsity scenario: the ambient dimension $K$ can be large, possibly larger than the sample size
$n$, but most entries of $\bpi$ are either equal to zero or very small.

Our goal is to investigate the statistical properties of the Maximum Likelihood Estimator (MLE), defined by
\begin{equation}\label{MLE}
\hat\bpi \in \argmin_{\bpi\in \Pi}\big\{-\frac{1}{n}\sum_{i=1}^n\log f_{\bpi}(\bX_i)\big\},
\end{equation}
where the minimum is computed over a suitably chosen subset $\Pi$ of $\BB_+^K$. In the present work, we will consider
sets $\Pi=\Pi_n(\mu)$, depending on a parameter $\mu>0$ and the sample $\{\bX_1,\ldots,\bX_n\}$, defined by
\begin{equation}\label{Pi}
\Pi_n(\mu) = \bigg\{\bpi\in \BB_+^K: \min_{i\in [n]}\sum_{j=1}^K \pi_j f_j(\bX_i)\geq \mu\bigg\}.
\end{equation}
Note that the objective function in \eqref{MLE} is convex and the same is true for set \eqref{Pi}. Therefore, the MLE
$\hat\bpi$ can be efficiently computed even for large $K$ by solving a problem of convex programming. To ease notation, very often,
we will omit the dependence of $\Pi_n(\mu)$ on $\mu$ and write $\Pi_n$ instead of $\Pi_n(\mu)$.

The quality of an estimator $\hat\bpi$ can be measured in various ways. For instance, one can consider
the Kullback-Leibler divergence
\begin{equation}
\KL(f^* ||f_{\hat\bpi}) =
    \begin{cases}
    \int_{\calX} f^*(\bx)\,\log \frac{f^*(\bx)}{f_{\hat\bpi}(\bx)}\,\nu(d\bx), &
            \text{ if } P^*\big(f^*(\bX)=0 \text{ and } f_{\hat\bpi}(\bX)>0\big)=0, \\
    +\infty, & \text{otherwise},
    \end{cases}
\end{equation}
which has the advantage of bypassing identifiability
issues. One can also consider the (well-specified) setting where $f^*=f_{\bbeta^*}$ for some $\bbeta^*\in\BB^K_+$
and measure the quality of estimation through a distance between the vectors $\hat\bpi$ and $\bpi^*$ (such as
the $\ell_1$-norm $\|\hat\bpi-\bpi^*\|_1$ or the Euclidean norm $\|\hat\bpi-\bpi^*\|_2$).

The main contributions of the present work are the following:
\begin{enumerate}[(a)]
\item We demonstrate that in the mixture model there is no need to introduce sparsity favoring penalty in order to
get optimal rates of estimation under the Kullback-Leibler loss in the sparsity scenario. In fact, the constraint
that the weight vector belongs to the simplex acts as a sparsity inducing penalty. As a consequence, there is no
need to tune a parameter accounting for the magnitude of the penalty.
\item We show that the maximum likelihood estimator of the mixture density simultaneously attains the optimal rate
of aggregation for the Kullback-Leibler loss for at least three types of aggregation: model-selection, convex and
$D$-sparse aggregation.
\item We introduce a new type of aggregation, termed \textit{nearly $D$-sparse aggregation} that extends and unifies
the notions of convex and $D$-sparse aggregation. We establish strong lower bounds for the nearly $D$-sparse aggregation
and demonstrate that the maximum likelihood estimator attains this lower bound up to logarithmic factors.
\end{enumerate}

\subsection{Related work}\label{ssec:rel_work}

The results developed in the present work aim to gain a better understanding (a) of the statistical properties of the
maximum likelihood estimator over a high-dimensional simplex and (b) of the problem of aggregation of density estimators
under the Kullback-Leibler loss. Various procedures of aggregation\footnote{We refer the interested reader to \citep{TsybICM}
for an up to date introduction into aggregation of statistical procedures.} for density estimation have been studied in the literature
with respect to different loss functions. \citep{Catoni97,Yang2000,JRT} investigated different variants of the progressive
mixture rules, also known as mirror averaging \citep{YNTV,DT12a}, with respect to the Kullback-Leibler loss and established
model selection type oracle inequalities\footnote{This means that they prove that the expected loss of the aggregate is
almost as small as the loss of the best element of the dictionary $\{f_1,\ldots,f_K\}$.} in expectation. Same type of guarantees,
but holding with high probability, were recently obtained in~\citep{Bellec2014,Butucea1} for the procedure termed $Q$-aggregation,
introduced in other contexts by \citep{Dai,Rigollet12}.

Aggregation of estimators of a probability density function under the  $L_2$-loss was considered in \citep{RT7}, where it was shown
that a suitably  chosen unbiased risk estimate minimizer is optimal both for convex and linear aggregation. The goal in the present work
is to go beyond the settings of the aforementioned papers in that we want simultaneously to do as well
as the best element of the dictionary, the best convex combination of the dictionary elements but also
the best sparse convex combination. Note that the latter task was coined $D$-aggregation in \citep{Lounici7}
(see also \citep{bunea2007}). In the present work, we rename it in $D$-sparse aggregation, in order to make
explicit its relation to sparsity.

Key differences between the latter work and ours are that we do not assume the sparsity index to be known and we are analyzing
an aggregation strategy that is computationally tractable even for large $K$. This is also the case of \citep{SPADES,Bertin},
which are perhaps the most relevant references to the present work. These papers deal with the $L_2$-loss and investigate the
lasso and the Dantzig estimators, respectively, suitably adapted to the problem of density estimation. Their methods handle dictionary
elements $\{f_j\}$ which are not necessarily probability density functions, but has the drawback of requiring the choice of a
tuning parameter. This choice is a nontrivial problem in practice. Instead, we show here that the optimal rates of sparse
aggregation with respect to the Kullback-Leibler loss can be attained by procedure which is tuning parameter free.

Risk bounds for the maximum likelihood and other related estimators in the mixture model have a long history \citep{LiB99,Li99,Rakhlin5}.
For the sake of comparison we recall here two elegant results providing non-asymptotic guarantees for the Kullback-Leibler loss.

\begin{theorem}[Theorem 5.1 in \citep{Li99}]\label{th:1}
Let $\calF$ be a finite dictionary of cardinality $K$ of density functions such that $\max_{f\in\calF}\|f^*/f\|_\infty\le V$. Then, the maximum likelihood estimator
over $\calF$, $\hat f^{\rm ML}_{\calF}\in{\rm arg}\max_{f\in\calF} \sum_{i=1}^n \log f(\bX_i)$, satisfies the inequality
\begin{equation}\label{eqthOne}
\Ex_{f^*}\big[\KL\big(f^*||\hat f^{\rm ML}_{\calF}\big)\big] \le \big(2+\log V\big)\bigg(\min_{f\in\calF} \KL(f^*|| f) + \frac{2\log K}{n}\bigg).
\end{equation}
\end{theorem}
Inequality \eqref{eqthOne} is an inexact oracle inequality in expectation that quantifies the ability of $\hat f^{\rm ML}_{\calF}$
to solve the problem of model-selection aggregation. The adjective inexact refers to the fact that the ``bias term'' $\min_{f\in\calF} \KL(f^*|| f)$
is multiplied by factor strictly larger than one. It is noteworthy that the remainder term $\frac{2\log K}{n}$ corresponds to the optimal rate
of model-selection aggregation \citep{Juditski00,Tsybakov03}. In relation with \Cref{th:1}, it is worth mentioning a result of \citep{Yang2000} and \citep{Catoni97},
see also Theorem~5 in \citep{Lecue06} and Corollary 5.4 in \citep{JRT}, establishing a risk bound similar to \eqref{eqthOne} without the
extra factor $2+\log V$ for the so called mirror averaging aggregate.

\begin{theorem}[page 226 in \citep{Rakhlin5}]\label{th:2}
Let $\calF$ be a finite dictionary of cardinality $K$ of density functions and let $\calC_k=\big\{f_{\bpi}:\|\bpi\|_0\le k\big\}$ be the set of all
the mixtures of at most $k$ elements of $\calF$ ($k\in[K]$). Assume that $f^*$ and the densities $f_k$ from $\calF$ are bounded from below and above
by some positive constants $m$ and $M$, respectively. Then, there is a constant $C$ depending only on $m$ and $M$ such that, for any tolerance level
$\delta\in(0,1)$, the maximum likelihood estimator
over $\calC_k$, $\hat f^{\rm ML}_{\calC_k}\in{\rm arg}\max_{f\in\calC_k} \sum_{i=1}^n \log f(\bX_i)$, satisfies the inequality
\begin{equation}
\KL\big(f^*||\hat f^{\rm ML}_{\calC_k}\big) \le \min_{f\in\calC_k} \KL(f^*|| f) + C\Big(\frac{\log(K/\delta)}{n}\Big)^{\nicefrac12}
\end{equation}
with probability at least $1-\delta$.
\end{theorem}

This result is remarkably elegant and can be seen as an exact oracle inequality in deviation for $D$-sparse aggregation (for $D=k$).
Furthermore, if we choose $k=K$ in \Cref{th:2}, then we get an exact oracle inequality for convex aggregation with a
rate-optimal remainder term  \citep{Tsybakov03}. However, it fails to provide the optimal rate for $D$-sparse aggregation.

Closing this section, we would like to mention the recent work \citep{Xia16}, where oracle inequalities for estimators of low rank
density matrices are obtained. They share a common feature with those obtained in this work: the adaptation to the unknown sparsity
or rank is achieved without any additional penalty term. The constraint that the unknown parameter belongs to the simplex acts
as a sparsity inducing penalty.

\subsection{Additional notation}

In what follows, for any $i\in [n]$, we denote by $\bZ_i$ the vector $[f_1(\bX_i),\ldots,f_K(\bX_i)]^\top$ and
by $\bfZ$ the $n\times K$ matrix  $[\bZ^\top_1,\ldots,\bZ^\top_n]^\top$. We also define $\ell(u) = -\log u$,
$u\in(0,+\infty)$,
so that the MLE $\hat\bpi$ is the minimizer of the function
\begin{equation}
\label{Phi}
L_n(\bpi) = \frac1n\sum_{i=1}^n \ell\big(\bZ_i^\top\bpi\big).
\end{equation}
For any set of indices $J\subseteq [K]$ and any $\bpi=(\pi_1,\dots,\pi_K)^{\top}\in\RR^K$,
we define $\bpi_J$ as the $K$-dimensional vector whose $j$-th coordinate equals $\pi_j$ if $j\in J$ and $0$ otherwise.
We denote the cardinality of any $J\subseteq[K]$ by $\vert J\vert$. For any set $J\subset \{1, \dots, K\}$ and any
constant $c\geq 0$, we introduce the compatibility constants \citep{VandeGeerConditionLasso09}
of a $K\times K$ positive semidefinite matrix $\bfA$,
\begin{align}
  \kappa_{\bfA}(J,c) &= \inf\bigg\{ \frac{c^2|J|\|\bfA^{\nicefrac12}\bv \|_2^2}{(c\|\bv_J\|_1-\|\bv_{J^c}\|_1)^2} :
        \bv\in\RR^K,\|\bv_{J^c}\|< c\|\bv_J\|_1\bigg\},\\
 \bar\kappa_\bfA(J,c) &=\inf\bigg\{\frac{|J|\|\bfA^{\nicefrac12}\bv\|_2^{2}}{\Vert \bv_{J}\Vert^{2}_{1}}:
        \bv\in\RR^{K},\Vert \bv_{J^{c}}\Vert_{1}<c\Vert \bv_{J}\Vert_{1}\bigg\}.
\end{align}
The risk bounds established in the present work involve the factors $\kappa_{\bfA}(J,3)$ and
$\bar\kappa_{\bfA}(J,1)$. One can easily check that $\bar\kappa_{\bfA}(J,3)\le \kappa_{\bfA}(J,3)
\le \frac94 \bar\kappa_{\bfA}(J,1)$. We also recall that the compatibility constants of a matrix
$\bfA$ are bounded  from below by the smallest eigenvalue of $\bfA$.

Let us fix a function $f_0:\calX\to\RR$ and denote
$\bar f_k = f_k-f_0$ and $\bar \bZ_i = [\bar f_1(\bX_i),\ldots,\bar f_K(\bX_i)]^\top$
for $i\in[n]$. In the results of this work, the compatibility factors are used for the empirical
and population Gram matrices of vectors $\bar\bZ_k$, that is when $\bfA = \hat\bSigma_n$ and
$\bfA = \bSigma$ with
\begin{equation}
\label{Gram}
\hat\bSigma_n = \frac1n \sum_{i=1}^n \bar\bZ_i\bar\bZ_i^\top,\qquad
\bSigma = \Ex[\bar\bZ_1\bar\bZ_1^\top].
\end{equation}
The general entries of these matrices are respectively $(\hat\bSigma_n)_{k,l} =
\nicefrac1n\sum_{i=1}^n \bar f_k(\bX_i)\bar f_l(\bX_i)$ and $(\bSigma)_{k,l} =
\Ex[\bar f_k(\bX_1)\bar f_l(\bX_1)]$.

We assume that there exist positive constants $m$ and $M$ such that for all densities
$f_k$ with $k \in [K]$, we have
\begin{equation}
\label{densConst}
  \forall x \in \calX, \quad m\le f_k(x) \le M.
\end{equation}
We use the notation $V = M/m$. It is worth mentioning that the set of dictionaries
satisfying simultaneously this boundedness assumption and the aforementioned compatibility
condition is not empty. For instance, one can consider the functions
$f_k(x) = 1+\nicefrac12\sin(2\pi k x)$ for $k\in [K]$. These functions are probability
densities w.r.t.\ the Lebesgue measure on $\calX = [0,1]$. They are bounded from below
and from  above by $\nicefrac12$ and $\nicefrac32$, respectively. Taking $f_0(x) = 1$,
the corresponding Gram matrix is $\bSigma = \nicefrac18\,\bfI_K$, which has all eigenvalues
equal to $\nicefrac18$.

\subsection{Agenda}

The rest of the paper is organized as follows. In \Cref{sec:main_results}, we state our
main theoretical contributions and discuss their consequences. Possible relaxations of
the conditions, as well as lower bounds showing the tightness of the established risk
bounds, are considered in \Cref{sec:discussion_of_the_results}.
A brief summary of the paper and some future directions of research are presented
in \Cref{sec:conclusion}. The proofs of all theoretical results are postponed
to \Cref{sec:proofs} and \Cref{sec:proof-lower}.

\section{Oracle inequalities in deviation and in expectation} 
\label{sec:main_results}

In this work, we prove several non-asymptotic risk bounds that
imply, in particular, that the maximum likelihood estimator is optimal in model-selection aggregation,
convex aggregation and $D$-sparse aggregation (up to $\log$-factors). In all the results of this section
we assume the parameter $\mu$ in \eqref{Pi} to be equal to $0$.
\begin{theorem}
\label{maintheo1}
Let $\cal F$ be a set of $K\ge 4$ densities satisfying the boundedness condition \eqref{densConst}.
Denote by $f_{\hat\bpi}$ the mixture density corresponding to the maximum likelihood estimator
$\hat\bpi$ over $\Pi_n$ defined in \eqref{Phi}. There are constants $c_1\le 32V^3$,
$c_2\le 288 M^2V^6$ and $c_3\le 128 M^2V^6$ such that, for any $\delta \in (0,\nicefrac12)$, the following
inequalities hold
\begin{align}
  \KL(f^*|| f_{\hat\bpi}) &\leq \inf_{\substack{J\subset [K]\\\bpi \in \BB_+^K}}
  \bigg\{ \KL(f^*|| f_{\bpi}) + c_1\Big(\frac{\log(K/\delta)}{n}\Big)^{\nicefrac12} \|\bpi_{J^c}\|_1 +
			\frac{c_2|J|\log(K/\delta)}{n\kappa_{\hat\bSigma_n}(J,3)} \bigg\},\label{boundDeviation}\\
 \KL(f^*|| f_{\hat\bpi}) &\leq \inf_{J\subset [K]}\inf_{\substack{\bpi \in \BB_+^K\\ \bpi_{J^c}=0}}
            \bigg\{ \KL(f^*|| f_{\bpi}) +\frac{c_3|J|\log(K/\delta)}{n\bar\kappa_{\hat\bSigma_n}(J,1)}
            \bigg\} \label{boundDevTwo}
\end{align}
with probability at least $1-\delta$.
\end{theorem}

The proof of this and the subsequent results stated in this section are postponed to \Cref{sec:proofs}.
Comparing the two inequalities of the above theorem, one can notice two differences. First, the term
proportional to $\|\bpi_{J^c}\|_1$  is absent in the second risk bound, which means that
the risk of the MLE is compared to that of the best mixture with a weight sequences supported by $J$.
Hence, this risk bound is weaker than the first one provided by \eqref{boundDeviation}. Second,
the compatibility factor $\bar\kappa_{\hat\bSigma_n}(J,1)$ in \eqref{boundDevTwo} is larger that
its counterpart $\kappa_{\hat\bSigma_n}(J,3)$ in \eqref{boundDeviation}. This entails that in the cases
where the oracle is expected to be sparse, the remainder term of the bound in \eqref{boundDeviation}
is slightly looser than that of \eqref{boundDevTwo}.

A first and simple consequence of \Cref{th:1} is obtained by taking $J=\varnothing$ in the right hand
side of the first inequality. Then, $\|\bpi_{J^c}\|_1=\|\bpi\|_1=1$ and we get
\begin{equation}
  \KL(f^*|| f_{\hat\bpi}) \leq \inf_{\bpi \in \BB_+^K}
  \KL(f^*|| f_{\bpi}) + c_1\Big(\frac{\log(K/\delta)}{n}\Big)^{\nicefrac12}.\label{convOracle}
\end{equation}
This implies that for every dictionary $\cal F$, without any assumption on the smallness of the coherence
between its elements, the maximum likelihood estimator achieves the optimal rate of convex aggregation, up
to a possible\footnote{In fact, the optimal rate of convex aggregation when $K\ge n^{\nicefrac12}$ is of
order $\normalsize\big(\nicefrac{\log(K/n^{\nicefrac12})}{\displaystyle n}\big)^{\nicefrac12}$. Therefore, even the
$\log K$ term is optimal whenever $K\ge C n^{\nicefrac12+\alpha}$ for some $\alpha>0$.} logarithmic
correction, in the high-dimensional regime $K\ge n^{\nicefrac12}$. In the case of regression with random
design, an analogous result has been proved by~\cite{LecueMend13} and \cite{Lecue13}.
One can also remark that the upper bound in \eqref{convOracle} is of the same form as the one of \Cref{th:2}
stated in~\cref{ssec:rel_work} above.

The main compelling feature  of our results is that they show that the MLE adaptively achieves the optimal rate of
aggregation not only in the case of convex aggregation, but also for the model-selection aggregation and $D$-(convex)
aggregation. For handling these two cases, it is more convenient to get rid of the presence of the compatibility
factor of the empirical Gram matrix $\hat\bSigma_n$.  The latter can be replaced by the compatibility factor of
the population Gram matrix, as stated in the next result.

\begin{theorem}
\label{maintheo2}
Let $\cal F$ be a set of $K$ densities satisfying the boundedness condition \eqref{densConst}.
Denote by $f_{\hat\bpi}$ the mixture density corresponding to the maximum likelihood estimator
$\hat\bpi$ over $\Pi_n$ defined in \eqref{Phi}. There are constants $c_4\le 32V^3 + 4$,
$c_5\le 4.5M^2(8\,V^3+1)^2$ and $c_6\le 2M^2(8\,V^3+1)^2$ such that, for any
$\delta \in (0,\nicefrac12)$, the following inequalities hold
\begin{align}
  \KL(f^*|| f_{\hat\bpi}) &\leq \inf_{\substack{J\subset [K]\\ \bpi \in \BB_+^K}}
		\bigg\{ \KL(f^*|| f_{\bpi}) + c_4\Big(\frac{\log(K/\delta)}{n}\Big)^{\nicefrac12} \|\bpi_{J^c}\|_1 +
			\frac{c_5|J|\log(K/\delta)}{n\kappa_{\bSigma}(J,3)} \bigg\},\label{boundDevThree}\\
 \KL(f^*|| f_{\hat\bpi}) &\leq \inf_{J\subset [K]}\inf_{\substack{\bpi \in \BB_+^K\\ \bpi_{J^c}=0}}
		\bigg\{ \KL(f^*|| f_{\bpi}) +
			\frac{c_6|J|\log(K/\delta)}{n\bar\kappa_{\bSigma}(J,1)} \bigg\}\label{boundDevFour}
\end{align}
with probability at least $1-2\delta$.
\end{theorem}

The main advantage of the upper bounds provided by \Cref{maintheo2} as compared with those of
\Cref{maintheo1} is that the former is deterministic, whereas the latter involves the
compatibility factor of the empirical Gram matrix which is random. The price to pay for getting
rid of randomness in the risk bound is the increased values of the constants $c_4$, $c_5$ and $c_6$.
Note, however, that this price is not too high, since obviously $1\le M\le L$ and, therefore,
$c_4\le 1.25 c_1$, $c_5\le 1.56 c_2$ and $c_6\le 1.56 c_3$. In addition, the absence of randomness
in the risk bound allows us to integrate it and to convert the bound in deviation into a bound
in expectation.

\begin{theorem}[Bound in Expectation]
\label{th:expectation}
Let $\cal F$ be a set of $K$ densities satisfying the boundedness condition \eqref{densConst}.
Denote by $f_{\hat\bpi}$ the mixture density corresponding to the maximum likelihood estimator
$\hat\bpi$ over $\Pi_n$ defined in \eqref{Phi}. There are constants $c_7\le 20V^3 + 8$,
$c_8\le M^2(22V^3+3)^2$ and $c_9\le M^2(15V^3+2)^2$ such that
\begin{align}
  \Ex[\KL(f^*|| f_{\hat\bpi})] &\leq \inf_{\substack{J\subset [K]\\ \bpi \in \BB_+^K}}
  \bigg\{ \KL(f^*|| f_{\bpi}) + c_7\Big(\frac{\log K}{n}\Big)^{\nicefrac12} \|\bpi_{J^c}\|_1 +
			\frac{c_8|J|\log K}{n\kappa_{\bSigma}(J,3)} \bigg\},\label{boundExpOne}\\
 \Ex[\KL(f^*|| f_{\hat\bpi})] &\leq \inf_{J\subset [K]}\inf_{\substack{\bpi \in \BB_+^K\\
		\bpi_{J^c}=0}} \bigg\{ \KL(f^*|| f_{\bpi}) +
			\frac{c_9|J|\log K}{n\bar\kappa_{\bSigma}(J,1)} \bigg\}.\label{boundExpTwo}
\end{align}
\end{theorem}

In inequality \eqref{boundExpTwo}, upper bounding the infimum over all sets $J$ by the infimum
over the singletons, we get
\begin{equation}
 \Ex[\KL(f^*|| f_{\hat\bpi})] \leq \inf_{j\in [K]}\bigg\{ \KL(f^*|| f_j) +
			\frac{c_9\log K}{n\bar\kappa_{\bSigma}(J,1)} \bigg\}.\label{MSaggr}
\end{equation}
This implies that the maximum likelihood estimator $f_{\hat\pi}$ achieves the
rate $\frac{\log K}{n}$ in model-selection type aggregation. This rate is known
to be optimal in the model of regression \citep{Rigollet12}. If we compare this result
with \Cref{th:1} stated in \Cref{ssec:rel_work}, we see that the remainder terms
of these two oracle inequalities are of the same order (provided that the compatibility
factor is bounded away from zero), but inequality \eqref{MSaggr} has the advantage
of being exact.

We can also apply  \eqref{boundExpTwo} to the problem of convex aggregation with
small dictionary, that is for $K$ smaller than  $n^{\nicefrac12}$. Upper bounding $|J|$ by $|K|$,
we get
\begin{equation}
 \Ex[\KL(f^*|| f_{\hat\bpi})] \leq \inf_{\bpi \in \BB_+^K}
		\KL(f^*|| f_{\bpi}) +\frac{c_9 K \log K}{n\bar\kappa_{\bSigma}([K],1)}.
		\label{Caggr}
\end{equation}
Assuming, for instance, the smallest eigenvalue of $\bSigma$ bounded away from zero (which
is a quite reasonable assumption in the context of low dimensionality), the above upper
bound provides a rate of convex aggregation of the order of $\frac{K\log K}{n}$. Up to a
logarithmic term, this rate is known to be optimal for convex aggregation in the model of
regression.

Finally, considering all the sets $J$ of cardinal smaller than $D$ (with $D\le K$) and
setting $\bar\kappa_{\bSigma}(D,1) = \inf_{J:|J|\le D} \bar\kappa_{\bSigma}(J,1)$, we
deduce from \eqref{boundExpTwo} that
\begin{equation}
  \Ex[\KL(f^*|| f_{\hat\bpi})] \leq \inf_{\bpi\in \BB_+^K : \|\bpi\|_0\le D}
  \KL(f^*|| f_{\bpi}) + \frac{c_9 D\log K}{n\bar\kappa_{\bSigma}(D,1)}.\label{Daggr}
\end{equation}
According to \citep[Theorem 5.3]{RT11}, in the regression model, the optimal
rate of $D$-sparse aggregation is of order $(D/n)\log(K/D)$, whenever $D=o(n^{\nicefrac12})$.
Inequality \eqref{Daggr} shows that the maximum likelihood estimator over the simplex
achieves this rate up to a logarithmic factor. Furthermore, this logarithmic inflation
disappears when the sparsity $D$ is such that, asymptotically, the ratio $\frac{\log D}{\log K}$
is bounded from above by a constant $\alpha<1$. Indeed, in such a situation the optimal
rate $\frac{D\log(K/D)}{n} = \frac{D\log K}{n}(1-\frac{\log D}{\log K})$ is of the same
order as the remainder term in \eqref{Daggr}, that is $\frac{D\log K}{n}$.
			

\section{Discussion of the conditions and possible extensions} 
\label{sec:discussion_of_the_results}

In this section, we start by announcing lower bounds for the Kullback-Leibler aggregation
in the problem of density estimation. Then we discuss the implication of the risk bounds of the
previous section to the case where the target is the weight vector $\bpi$ rather than the
mixture density $f_{\bpi}$. Finally, we present some extensions to the case where the boundedness
assumption is violated.

\subsection{Lower bounds for nearly-$D$-sparse aggregation}
\label{ssec:lower}

As mentioned in previous section, the literature is replete with lower bounds on the minimax
risk for various types of aggregation. However most of them concern the regression setting either
with random or with deterministic design. Lower bounds of aggregation for density estimation were
first established by \cite{rigollet:these} for the $L_2$-loss. In the case of Kullback-Leibler
aggregation in density estimation, the only lower bounds we are aware are those established
by \cite{Lecue06} for model-selection type aggregation. It is worth emphasizing here that the
results of the aforementioned two papers provide weak lower bounds. Indeed, they establish
the existence of a dictionary for which the minimax excess risk is lower bounded by the
suitable quantity. In contrast with this, we establish here strong lower bounds that hold
for every dictionary satisfying the boundedness and the compatibility conditions.

Let $\calF = \{f_1,\ldots,f_K\}$ be a dictionary of density functions on $\calX =[0,1]$.
We say that the dictionary $\calF$ satisfies the boundedness and the compatibility assumptions
if for some positive constants $m,M$ and $\kappa$, we have $m \le f_j(x)\le M $ for all
$j\in[K]$, $x\in\calX$. In addition,  we assume in this subsection that all the eigenvalues
of the Gram matrix $\bSigma$ belong to the interval $[\varkappa_*,\varkappa^*]$, with
$\varkappa_*>0$ and $\varkappa^*<\infty$.

For every $\gamma \in (0,1)$  and any $D\in [K]$, we define the set of nearly-$D$-sparse
convex combinations of the dictionary elements $f_j\in\calF$ by
\begin{equation}
\calH_\calF(\gamma,D) = \Big\{f_{\bpi} : \bpi\in\BB^K_+ \text{ such that }\exists \, J\subset [K]
\text{ with } \|\bpi_{J^c}\|_1\le \gamma \text{ and } |J|\le D\Big\}.
\end{equation}
In simple words, $f_{\bpi}$ belongs to $\calH_\calF(\gamma,D)$ if it admits a $\gamma$-approximately
$D$-sparse representation in the dictionary $\calF$. We are interested in bounding  from
below the minimax excess risk
\begin{equation}
\calR\big(\calH_\calF(\gamma,D)\big) = \inf_{\hat f} \sup_{f^*}
\Big\{\Ex[\KL(f^*||\,\hat f\,)] - \inf_{f_{\bpi}\in \calH_\calF(\gamma,D)} \KL(f^*|| f_{\bpi})\Big\},
\end{equation}
where the $\inf$ is over all possible estimators of $f^*$ and the $\sup$ is over all density functions
over $[0,1]$. Note that the estimator $\hat f$ is not necessarily a convex combination of the
dictionary elements. Furthermore, it is allowed to depend on the parameters $\gamma$ and $D$
characterizing the class $\calH_\calF(\gamma,D)$. It follows from \eqref{boundExpOne}, that if
the dictionary satisfies the boundedness and the compatibility condition, then
\begin{equation}\label{upper}
\calR\big(\calH_\calF(\gamma,D)\big) \le C \Big\{\Big(\frac{\gamma^2\log K}{n}\Big)^{\nicefrac12} + \frac{D\log K}{n}\Big\}\bigwedge \Big(\frac{\log K}{n}\Big)^{\nicefrac12},
\end{equation}
for some constant $C$ depending only on $m,M$ and $\varkappa_*$. Note that the last term
accounts for the following phenomenon: If the sparsity index $D$ is larger than a multiple
of $\sqrt{n}$, then the sparsity bears no advantage as compared to the $\ell_1$ constraint.
The next result implies that this upper bound is optimal, at least up to logarithmic
factors.

\begin{theorem}
\label{theorem:lower_bound}
Assume that $\log(1+eK)\le n$. Let $\gamma\in(0,1)$ and $D\in[K]$ be fixed. There exists
a constant $A$ depending only on $m$, $M$, $\varkappa_*$ and  $\varkappa^*$ such that
\begin{equation}
\calR(\calH_\calF(\gamma,D)) \ge A \bigg\{\bigg[\frac{\gamma^2}{n}
\log\bigg(1+\frac{K}{\gamma\sqrt{n}}\bigg)\bigg]^{\nicefrac12}
+ \frac{D\log(1+K/D)}{n}\bigg\}\bigwedge \bigg[\frac{1}{n}
\log\bigg(1+\frac{K}{\sqrt{n}}\bigg)\bigg]^{\nicefrac12}.
\end{equation}
\end{theorem}

This is the first result providing lower bounds on the minimax risk of aggregation over
nearly-$D$-sparse aggregates. To the best of our knowledge, even in the Gaussian sequence
model, such a result has not been established to date. It has the advantage of
unifying the results on convex and $D$-sparse aggregation, as well as extending them
to a more general class. Let us also stress that the condition $\log(1+eK)\le n$
is natural and unavoidable, since it ensures that the right hand side of \eqref{upper}
is smaller than the trivial bound $\log V$.

\subsection{Weight vector estimation}
\label{ssec:weight}

The risk bounds carried out in the previous section for the problem of
density estimation in the Kullback-Leibler loss imply risk bounds for
the problem of weight vector estimation. Indeed, under the boundedness
assumption \eqref{densConst}, the Kullback-Leibler divergence between
two mixture densities can be shown to be equivalent to the squared
Mahalanobis distance between the weight vectors of these mixtures with
respect to the Gram matrix.  In order to go from the Mahalanobis distance
to the Euclidean one, we make use of the restricted eigenvalue
\begin{equation}
\kappa^{\rm RE}_{\bSigma}(s,c)=\inf\big\{\| \bSigma^{\nicefrac12}\bv\|_2^{2}:\,
\exists\,J\subset [K] \text{ s.t. } |J|\le s,\
\|\bv_{J^{c}}\|_{1}\le c\|\bv_{J}\|_{1}\ \text{and}\ \|\bv_{J}\|_{2}=1\big\}.
\end{equation}
This strategy leads to the next result.

\begin{proposition}\label{prop:1}
Let $\cal F$ be a set of $K\ge 4$ densities satisfying condition \eqref{densConst}.
Denote by $f_{\hat\bpi}$ the mixture density corresponding to the maximum likelihood
estimator $\hat\bpi$ over $\Pi_n$ defined in \eqref{Phi}. Let $\bpi^*$ the weight-vector
of the best mixture density: $\bpi^*\in\text{\rm arg}\min_{\bpi} \KL(f^*||f_{\bpi})$,
and let $J^*$ be the support of $\bpi^*$. There are constants $c_{10}\le M^2(64V^3+8)$
and $c_{11}\le 4M^2(8V^3+1)$ such that, for any
$\delta \in (0,\nicefrac12)$, the following  inequalities hold
\begin{align}
 \|\hat\bpi - \bpi^*\|_1
			&\leq  \frac{c_{10}|J^*|}{\bar\kappa_{\bSigma}(J^*,1)}\,\Big(\frac{\log(K/\delta)}{n}\Big)^{\nicefrac12},
					\label{elOne}\\
\|\hat\bpi - \bpi^*\|_2
			&\leq  \frac{c_{11}}{\kappa^{\rm RE}_{\bSigma}(|J^*|,1)}\,\Big(\frac{2|J^*|\log(K/\delta)}{n}\Big)^{\nicefrac12},
					\label{euclOne}\\		
\|\hat\bpi - \bpi^*\|_2^2
			&\leq  \frac{c_{11}}{\kappa^{\rm RE}_{\bSigma}(|J^*|,1)}\,\Big(\frac{2\log(K/\delta)}{n}\Big)^{\nicefrac12}
			\label{euclTwo}
\end{align}
with probability at least $1-2\delta$.
\end{proposition}

In simple words, this result tells us that the wight estimator $\hat\bpi$  attains the
minimax rate of estimation $|J^* |(\frac{\log(K)}{n})^{\nicefrac12}$ over the intersection of the
$\ell_1$ and $\ell_0$ balls, when the error is measured by the $\ell_1$-norm, provided
that the compatibility factor of the dictionary $\calF$ is bounded away from zero.
The optimality of this rate---up to logarithmic factors---follows from the fact that the
error of estimation of each nonzero coefficients of $\bpi^*$ is at least $cn^{-\nicefrac12}$ (for
some $c>0$), leading to a sum of the absolute values of the errors at least of the order
$|J^*|n^{-\nicefrac12}$. The logarithmic inflation of the rate is the price to pay for not knowing
the support $J^*$. It is clear that this reasoning is valid only when the sparsity
$|J^*|$ is of smaller order than $n^{\nicefrac12}$. Indeed, in the case $|J^*|\ge c n^{\nicefrac12}$,
the trivial bound $\|\hat\bpi-\bpi^*\|_1\le 2$ is tighter than the one in \eqref{elOne}.

Concerning the risk measured by the Euclidean norm, we underline that there are two regimes
characterized by the order between upper bounds in \eqref{euclOne} and \eqref{euclTwo}.
Roughly speaking, when the signal is highly sparse in the sense that $|J^*|$ is
smaller than $(n/\log K)^{\nicefrac12}$, then the smallest bound is given by \eqref{euclOne} and
is of the order $\frac{|J^*|\log(K)}{n}$. This rate is can be compared to the
rate $\frac{|J^*|\log(K/|J^*|)}{n}$, known to be optimal in the Gaussian sequence model.
In the second regime corresponding to mild sparsity, $|J^*| > (n/\log K)^{\nicefrac12}$,
the smallest bound is the one in \eqref{euclTwo}. The latter is of order
$(\frac{\log(K)}{n})^{\nicefrac12}$, which is known to be optimal in the Gaussian sequence
model. For various results providing lower bounds in regression framework we refer
the interested reader to \citep{Rask11,RT11,Wang14}.

\subsection{Extensions to the case of vanishing components}
\label{ssec:vanish}

In the previous sections we have deliberately avoided any discussion of the role of the
parameter $\mu$, present in the search space $\Pi_n(\mu)$ of the problem \eqref{MLE}-\eqref{Pi}.
In fact, when all the dictionary elements are separated from zero by  a constant $m$, a
condition assumed throughout previous sections, choosing any value of $\mu\le m$ is equivalent
to choosing $\mu=0$. Therefore, the choice of this parameter does not impact the quality of
estimation. However, this parameter might have strong influence in practice both on statistical
and computational complexity of the maximum likelihood estimator. A first step in understanding
the influence of $\mu$ on the statistical complexity is made in the next paragraphs.

Let us consider the case where the condition $\min_x \min_j f_j(x)\ge m>0$ fails,
but the upper-boundedness condition $\max_x\max_j f_j(x)\le M$ holds true. In such a
situation, we replace the definition $V= M/m$ by $V=M/\mu$. We also define the set
$\Pi^*(\mu) = \big\{\bpi\in\BB^K_+ : P^*\big(f_{\bpi}(\bX)\ge \mu\big)=1\big\}$.
In order to keep mathematical formulae simple, we will only state the equivalent of
\eqref{boundDevTwo} in the case of $m=0$. All the other results of the previous
section can be extended in a similar way.

\begin{proposition}\label{prop:2}
Let $\cal F$ be a set of $K\ge 2$ densities satisfying the boundedness condition
$\sup_{\bx\in\calX} f_j(\bx)\le M$. Denote by $f_{\hat\bpi}$ the mixture density
corresponding to the maximum likelihood estimator $\hat\bpi$ over $\Pi_n(\mu)$
defined in \eqref{Phi}. There is a constant $\bar c\le 128 M^2V^4$ such that,
for any $\delta \in (0,\nicefrac12)$,
\begin{align}
 \KL(f^*|| f_{\hat\bpi}) &\leq \inf_{J\subset [K]}
    \inf_{\substack{\bpi \in \Pi^*(\mu)\\ \bpi_{J^c}=0}}\bigg\{ \KL(f^*|| f_{\bpi})
    + \frac{\bar c|J|\log(K/\delta)}{n\bar\kappa_{\hat\bSigma_n}(J,1)}\bigg\} +
    \int_{\calX}(\log\mu-\log f_{\hat\bpi})_+f^*d\nu
    \label{boundDevFive}
\end{align}
on an event of probability at least $1-\delta$. Furthermore, if $\inf_{\bx\in\calX} f^*(\bx)
\ge \mu$, then, on the same event, we have
\begin{align}
 \|f^*-f_{\hat\bpi}\|_{L^2(P^*)}^2
    &\leq 2M^2\inf_{J\subset [K]}\inf_{\substack{\bpi \in \Pi^*(\mu)\\ \bpi_{J^c}=0}}
        \bigg\{ \KL(f^*|| f_{\bpi}) + \frac{\bar c|J|\log(K/\delta)}{
        n\bar\kappa_{\hat\bSigma_n}(J,1)}\bigg\}.
    \label{boundDevSix}
\end{align}
\end{proposition}

The last term present in the first upper bound,  $\int_{\calX}(\log\mu-\log f_{\hat\bpi})_+f^*d\nu$
is the price we pay for considering densities that are not lower bounded by a given constant.
A simple, non-random upper bound on this term is $\int_{\calX}\max_{k\in [K]}(\log\mu-\log f_k)_+f^*d\nu$.
Providing a tight upper bound on this kind or remainder terms is an important problem which lies
beyond the scope of the present work.



\section{Conclusion} 
\label{sec:conclusion}

In this paper, we have established exact oracle inequalities for the maximum likelihood
estimator of a mixture density. This oracle inequality clearly highlights the interplay
of three sources of error: misspecification of the model of mixture, departure from
$D$-sparsity and stochastic error of estimating $D$ nonzero coefficients.  We have also
proved a lower bound that show that the remainder terms of our upper bounds are optimal,
up to logarithmic terms. This lower bound is valid not only for the maximum likelihood
estimator, but for any estimator of the density function. As a consequence, the maximum
likelihood estimator has a nearly optimal excess risk in the minimax sense.

In all the results of the present paper, we have assumed that the components of the mixture
model are deterministic. From a practical point of view, it might be reasonable to choose
these components in a data driven way, using, for instance, a hold-out sample. This question,
as well as the problem of tuning the parameter $\mu$ , constitute interesting and challenging
avenues for  future research.


\section{Proofs of results stated in previous sections} 
\label{sec:proofs}

This section collects the proofs of the theorems and claims stated in previous sections.

\subsection{Proof of \Cref{maintheo1}}

The main technical ingredients of the proof are a strong convexity argument and a control of
the maximum of an empirical process. The corresponding results are stated in \Cref{convexlemma}
and \Cref{boundEmpProcess}, respectively, deferred to  \Cref{ssec:auxiliary}. We denote by
$\bar\bfZ$ the $n\times K$ matrix $[\bar\bZ_1,\ldots,\bar\bZ_K]$.

Since $\hat\bpi$ is a minimizer of $L_n(\cdot)$, see \eqref{MLE} and \eqref{Phi}, we know that
$L_n(\hat\bpi)\le L_n(\bpi)$ for every $\pi$. However, this inequality can be made sharper using the
(local) strong convexity of the function $\ell(u) = -\log(u)$. Indeed, \Cref{convexlemma} below
shows that
\begin{equation}\label{ineqOne}
    \frac{1}{n}\sum_{i=1}^n \ell(f_{\hat\bpi}(\bX_i)) \leq
    \frac{1}{n}\sum_{i=1}^n \ell(f_{\bpi}(\bX_i)) - \frac{1}{2M^2n}\|\bar\bfZ(\hat\bpi-\bpi)\|^2_2.
\end{equation}
On the other hand, if we set $\varphi(\pi,\bx) = \int (\log f_{\bpi})f^*d\nu -\log f_{\bpi}(\bx)$,
we have $\Ex_{f^*}[\varphi(\bpi,\bX_i)]=0$ and
\begin{equation}\label{ineqTwo}
\ell(f_{\bpi}(\bX_i)) = \KL(f^*||f_{\bpi}) -\int_{\mathcal X} f^*\,\log f^*d\nu+\varphi(\pi,\bX_i).
\end{equation}
Combining inequalities \eqref{ineqOne} and \eqref{ineqTwo}, we get
\begin{equation}
\label{tempOraclIneq}
\KL(f^*|| f_{\hat\bpi}) \leq \KL(f^*|| f_{\bpi})-\frac{1}{2M^2n}\|\bar\bfZ(\hat\bpi-\bpi)\|_2^2
+ \frac1n\sum_{i=1}^n \big(\varphi(\bpi,\bX_i)-\varphi(\hat\bpi,\bX_i)\big).
\end{equation}
The next step of the proof consists in establishing a suitable upper bound on the noise term
$\Phi_n(\bpi)-\Phi_n(\hat\bpi)$
where
\begin{equation}
\Phi_n(\bpi) = \frac1n\sum_{i=1}^n \varphi(\bpi,\bX_i).
\end{equation}
According to the mean value theorem, setting  $\zeta_n:= \sup_{\bar\bpi\in\bPi_n}\big\| \nabla \Phi_n(\bar\bpi)\big\|_\infty$,
for every vector $\bpi\in \bPi_n$,
it holds that
\begin{equation}
  |\Phi_n(\hat\bpi)-\Phi_n(\bpi)|\leq \sup_{\bar\bpi\in\bPi_n}\big\| \nabla \Phi_n(\bar\bpi)\big\|_\infty \|\hat\bpi-\bpi \|_1 =
  \zeta_n \|\hat\bpi-\bpi \|_1.
\end{equation}
This inequality, combined  with \eqref{tempOraclIneq}, yields
\begin{equation}
\label{midIneqRes}
\KL(f^*|| f_{\hat\bpi}) \leq \KL(f^*|| f_{\bpi})-\frac{1}{2M^2n}\|\bar\bfZ(\hat\bpi-\bpi)\|_2^2
+ \zeta_n\|\hat\bpi - \bpi \|_1.
\end{equation}
Using the Gram matrix $\hat\bSigma_n = \nicefrac1n\bar\bfZ^\top\bar\bfZ$, the quantity
$\|\bar\bfZ(\hat\bpi - \bpi)\|_2$ can be rewritten as
\begin{equation}
\label{gramRewrite}
  \|\bar\bfZ(\hat\bpi - \bpi)\|_2^2 = n\|\hat\bSigma_n^{\nicefrac12} (\hat\bpi - \bpi)\|_2^2.
\end{equation}

We proceed with applying the following result \citep[Lemma 2]{BDGP}.

\begin{lemma}[\cite{BDGP}, Lemma 2]
For any pair of vectors $\bpi,\bpi'\in\RR^K$, for any pair of scalars $\mu>0$ and $\gamma>1$, for any $K\times K$
symmetric matrix $\bfA$ and for any set $J \subset [p]$, the following inequality is true
\begin{equation}
   2\mu\gamma^{-1}(\| \bpi-\hat\bpi\|_1 + \gamma \|\bpi\|_1 - \gamma \|\hat\bpi\|_1 )-\|\bfA(\bpi-\hat\bpi)\|_2^2 \leq
  4\mu \| \bpi_{J^c}\|_1 + \frac{(\gamma+1)^2\mu^2|J|}{\gamma^2\kappa_{\bfA^2}(J,c_{\gamma})},
\end{equation}
where $c_{\gamma}=(\gamma+1)/(\gamma-1)$.
\end{lemma}

Choosing $\bfA = \hat\bSigma_n^{\nicefrac12}/(\sqrt{2}\,M)$, $\mu=\zeta_n$ and $\gamma=2$ (thus $c_\gamma = 3$)
we get the inequality
\begin{equation}
  \zeta_n\| \bpi-\hat\bpi\|_1-\|\bfA(\bpi-\hat\bpi)\|_2^2 \leq
  4\zeta_n\| \bpi_{J^c}\|_1 + \frac{9\zeta_n^2|J|}{4\kappa_{\bfA^2}(J,3)},\qquad \forall J \in \{1, \dots, p\}.
\end{equation}
One can check that $\kappa_{\bfA^2}(J,3) = \kappa_{\hat\bSigma_n}(J,3)/(2M^2)$.
Combining the last inequality with \eqref{midIneqRes}, we arrive at
\begin{equation}
\label{oracleIneqTh}
  \KL(f^*|| f_{\hat\bpi}) \leq \KL(f^*|| f_{\bpi}) + 4\zeta_n\| \bpi_{J^c}\|_1 + \frac{9M^2\zeta_n^2|J|}{2\kappa_{\hat\bSigma_n}(J,3)}.
\end{equation}
Since the last inequality holds for every $\bpi$, we can insert an $\inf_{\bpi}$ in the right hand side.
Furthermore, in view of \Cref{boundEmpProcess} below, with probability larger than $1-\delta$, $\zeta_n$ is bounded from above
by $8V^3(\frac{\log(K/\delta)}{n})^{\nicefrac12}$. This completes the proof of \eqref{boundDeviation}.

To prove \eqref{boundDevTwo}, we follow the same steps as above up to inequality \eqref{midIneqRes}.
Then, we remark that for every $\bpi$ in the simplex satisfying $\bpi_{J^c}=0$, it holds
\begin{equation}\label{inqCone}
\|(\hat\bpi-\bpi)_{J^c}\|_1= \|\hat\bpi_{J^c}\|_1 = 1-\|\hat\bpi_{J}\|_1 =
\|\bpi_{J}\|_1-\|\hat\bpi_{J}\|_1\le \|(\hat\bpi-\bpi)_{J}\|_1.
\end{equation}
Therefore, $\|\hat\bSigma_n^{\nicefrac12}(\hat\bpi-\bpi)\|_2^2\ge $ we have with probability at least $1-\delta$
\begin{align}
  \zeta_n\| \hat\bpi - \bpi\|_1 -\frac{1}{2M^2n}\|\bfZ(\hat\bpi-\bpi)\|_2^2
        &\leq 2\zeta_n\| (\hat\bpi-\bpi)_J\|_1-\frac{1}{2M^2}\|\hat\bSigma_n^{\nicefrac12}(\hat\bpi-\bpi)\|_2^2\\
        &\leq 2\zeta_n\| (\bpi-\hat\bpi)_J\|_1 - \frac{\bar\kappa_{\hat\bSigma_n}(J,1)\| (\bpi-\hat\bpi)_J\|_1^2}{2M^2|J|}\\
        &\leq \frac{2\zeta_n^2M^2|J|}{\bar\kappa_{\hat\bSigma_n}(J,1)}.\label{ineqa}
\end{align}
Replacing the right hand term in \eqref{midIneqRes} and taking the infimum, we get
the claim of the corollary. Since, in view of \Cref{boundEmpProcess} below, with probability
larger than $1-\delta$, $\zeta_n$ is bounded from above by
$8V^3(\frac{\log(K/\delta)}{n})^{\nicefrac12}$, we get the claim of \eqref{boundDevTwo}.

\subsection{Proof of \Cref{maintheo2}}
\label{ssec:proof:smallK}
Let us denote $\bv = \hat\bpi - \bpi$. According to \eqref{midIneqRes} and \eqref{gramRewrite},
we have
\begin{align}
\KL(f^*|| f_{\hat\bpi})
        &\leq \KL(f^*|| f_{\bpi})+ \zeta_n\|\hat\bpi - \bpi \|_1-
                \frac{1}{2M^2}\|\hat\bSigma_n^{\nicefrac12}(\hat\bpi-\bpi)\|_2^2\\
        &\leq \KL(f^*|| f_{\bpi})+ \zeta_n\|\bv\|_1-\frac{1}{2M^2}\|\bSigma^{\nicefrac12}\bv\|_2^2+
                \frac{1}{2M^2}\,\bv^\top(\bSigma-\hat\bSigma_n)\bv.
        \label{midIneqResTwo}
\end{align}
As $\bv$ is the difference of two vectors lying on the simplex, we have $\|\bv\|_1\le 2$. Let
$\|\bSigma-\hat\bSigma_n\|_\infty = \max_{j,j'}|(\bSigma-\hat\bSigma_n)_{j,j'}|$ stand for the
largest (in absolute values) element of the matrix $\bSigma-\hat\bSigma_n$. We have
\begin{equation}
\bv^\top(\bSigma-\hat\bSigma_n)\bv \le \|\bSigma-\hat\bSigma_n\|_{\infty} \|\bv\|_1^2
\le 2\|\bSigma-\hat\bSigma_n\|_{\infty} \|\bv\|_1.
\end{equation}
Setting $\bar\zeta_n = \zeta_n + M^{-2}\|\bSigma-\hat\bSigma_n\|_{\infty}$, we get
\begin{align}
\KL(f^*|| f_{\hat\bpi})
        &\leq \KL(f^*|| f_{\bpi})+ \bar\zeta_n\|\hat\bpi - \bpi \|_1-
                \frac{1}{2M^2}\|\bSigma^{\nicefrac12}(\hat\bpi-\bpi)\|_2^2.
                \label{midIneqResThree}
\end{align}
Following the same steps as those used for obtaining \eqref{oracleIneqTh}, we arrive at
\begin{equation}
\label{oracleIneqThTwo}
  \KL(f^*|| f_{\hat\bpi}) \leq \KL(f^*|| f_{\bpi}) + 4\bar\zeta_n\| \bpi_{J^c}\|_1 +
  \frac{9\bar\zeta_n^2M^2|J|}{2\kappa_{\bSigma}(J,3)}.
\end{equation}
The last step consists in evaluating the quantiles of the random variable $\bar\zeta_n$.
To this end, one checks that the Hoeffding inequality combined with the union bound yields
\begin{equation}
\Pb\Big\{\|\bSigma-\hat\bSigma_n\|_{\infty} >t\Big\} \le K(K-1) \exp({-2nt^2/M^4}),\qquad\forall t>0.
\end{equation}
In other terms, for every $\delta\in(0,1)$, we have
\begin{equation}
\Pb\Big\{\|\bSigma-\hat\bSigma_n\|_{\infty} \le M^2\Big(\frac{\log(K^2/\delta)}{2n}\Big)^{\nicefrac12} \Big\}
\ge 1-\delta.\label{SigmaInfty}
\end{equation}
Note that for $\delta\le 1$, we have $\log(K^2/\delta)\le 2\log(K/\delta)$. Combining with
\Cref{boundEmpProcess}, this implies that $\bar\zeta_n \le (8V^3+1)\big(\frac{\log(K/\delta)}{n}
\big)^{\nicefrac12}$ with probability larger than $1-2\delta$. This completes the proof of \eqref{boundDevThree}.
The proof of \eqref{boundDevFour} is omitted since it repeats the same arguments as those
used for proving \eqref{boundDevTwo}.

\subsection{Proof of \Cref{th:expectation}}

According to \eqref{oracleIneqThTwo}, for any $\bpi\in \Pi$ and
any $J\subset\{1,\dots,K\}$, we have
\begin{equation}\label{eqSixThreeOne}
  \Ex[\KL(f^*|| f_{\hat\bpi})] \leq \KL(f^*|| f_{\bpi}) + 4\| \bpi_{J^c}\|_1\Ex[\bar\zeta_n] +
  \frac{9 M^2|J|}{2\kappa_{\bSigma}(J,3)}\,\Ex[\bar\zeta_n^2].
\end{equation}
Recall now that $\bar\zeta_n = \zeta_n+M^{-2}\|\hat\bSigma_n-\bSigma\|_{\infty}$ and, according to
\Cref{boundEmpProcess}, we have
\begin{equation}\label{eqSixThreeTw}
  \Ex[\zeta_n] \leq 4V^3\Big(\frac{2\log(2K^2)}{n}\Big)^{\nicefrac12} \qquad\text{and}\qquad
  \var[\zeta_n] \leq \frac{V^2}{2n}.
\end{equation}
Using \Cref{Hoeffding1}, one easily checks that
\begin{equation}\label{eqSixThreeTh}
  \Ex[\|\hat\bSigma_n-\bSigma\|_{\infty}] \leq M^2\Big(\frac{\log(2K^2)}{2n}\Big)^{\nicefrac12}.
\end{equation}
This implies that
\begin{equation}\label{eqSixThreeF}
  \Ex[\bar\zeta_n] \leq (8V^3+1\big)\Big(\frac{\log(2K^2)}{2n}\Big)^{\nicefrac12}.
\end{equation}
Similarly, in view of the Efron-Stein inequality, we have
$\var[\|\hat\bSigma_n-\bSigma\|_{\infty}]\le \frac{M^4}{2n}$. This implies that
\begin{align}
  \Ex[\bar\zeta_n^2]
    &\leq (\Ex[\bar\zeta_n])^2 + \big\{(\var[\zeta_n])^{\nicefrac12}+
			M^{-2}(\var[\|\hat\bSigma_n-\bSigma\|_{\infty}])^{\nicefrac12}\big\}^2 \\
    &\le (8V^3+1\big)^2\frac{\log(2K^2)}{2n} + \frac{(V+1)^2}{2n}\\
    &\le 1.615(8V^3+1\big)^2\frac{\log K}{n}.\label{eqSixThreeFv}
\end{align}
Combining \eqref{eqSixThreeF}, \eqref{eqSixThreeFv} and \eqref{eqSixThreeOne}, we get the desired result.

\subsection{Proof of \Cref{prop:1}}

Using the strong convexity of the function $u\mapsto \log u$ over the interval $[m,M]$ and
the fact that $\bpi^*$ minimizes the convex function $\bpi\mapsto\KL(f^*||f_{\bpi})$, we get
\begin{equation}
\KL(f^*||f_{\hat\bpi}) \ge \KL(f^*||f_{\bpi^*})+\frac1{2M^2} \|\hat\bSigma_n^{\nicefrac12}(\hat\bpi-
\bpi^*)\|_2^2.
\end{equation}
Combining with \eqref{midIneqResThree}, in which we replace $\bpi$ by $\bpi^*$, we get
\begin{equation}\label{midIneqResFr}
\|\bSigma^{\nicefrac12}(\hat\bpi-\bpi^*)\|_2^2 \le 2 M^2\bar\zeta_n\|\hat\bpi-\bpi^*\|_1.
\end{equation}
Let us set $\bv = \hat\bpi-\bpi^*$. If $\bv=0$, then the claims are trivial. In the rest of this
proof, we assume $\|\bv\|_1>0$. In view of \eqref{inqCone}, we have
$\|\bv\|_1\le 2\|\bv_{J^*}\|_1$. Therefore, using the definition of the compatibility
factor, we get
\begin{equation}\label{vOne}
\|\bv\|_1^2\le 4\|\bv_{J^*}\|_1^2\le \frac{4|J^*|\,\|\bSigma^{\nicefrac12}\bv\|_2^2}
{\bar\kappa(J^*,1)} \le \frac{8|J^*|\,M^2\bar\zeta_n\|\bv\|_1}{\bar\kappa(J^*,1)}.
\end{equation}
We have already checked that $\bar\zeta_n \le (8V^3+1)\big(\frac{\log(K/\delta)}{n}
\big)^{\nicefrac12}$ with probability larger than $1-2\delta$. Dividing both sides of inequality
\eqref{vOne} by $\|\bv\|_1$ and using the aforementioned upper bound on $\bar\zeta_n$,
we get the desired bound on $\|\bv\|_1=\|\hat\bpi-\bpi^*\|_1$.

In order to bound the error $\bv=\hat\bpi-\bpi^*$ in the Euclidean norm, we denote
by $\hat J $ the set of $D = |J^*|$ indices corresponding to $D$ largest entries of the
vector $(|v_1|,\ldots,|v_K|)$. Since $\|\bv\|_1\le 2\|\bv_{J^*}\|_1$, we clearly have
$\|\bv\|_1\le 2\|\bv_{\hat J}\|_1$. Therefore,
\begin{align}
\|\bv\|_2^2
		&= \|\bv_{\hat J}\|_2^2+\|\bv_{\hat J^c}\|_2^2\\
		&\le \|\bv_{\hat J}\|_2^2+\|\bv_{\hat J^c}\|_\infty\|\bv_{\hat J^c}\|_1\\
		&\le \|\bv_{\hat J}\|_2^2+\frac{\|\bv_{\hat J}\|_1}{D}\|\bv_{\hat J^c}\|_1\\
		&\le \|\bv_{\hat J}\|_2^2+\frac{1}{D}\|\bv_{\hat J}\|_1^2\le 2\|\bv_{\hat J}\|_2^2
		\label{eqTwoEight}.
\end{align}
Combining this inequality with the definition of the restricted eigenvalue and
inequality \eqref{midIneqResFr} above, we arrive at
\begin{align}
\|\bv_{\hat J}\|_2^2\le \frac{\|\bSigma^{\nicefrac12}\bv\|_2^2}{\kappa^{\rm RE}(D,1)}
		\le  \frac{2M^2\bar\zeta_n\|\bv\|_1}{\kappa^{\rm RE}(D,1)}
		\le  \frac{4M^2\bar\zeta_n(\|\bv_{\hat J}\|_1\wedge 1)}{\kappa^{\rm RE}(D,1)}
		\le  \frac{4M^2\,\bar\zeta_n(\sqrt{D}\|\bv_{\hat J}\|_2\wedge 1)}{\kappa^{\rm RE}(D,1)}.
\end{align}
Dividing both sides by $\|\bv_{\hat J}\|_2$, taking the square and using \eqref{eqTwoEight},
we get
\begin{align}
\|\bv\|_2\le \sqrt2\,\|\bv_{\hat J}\|_2
		\le  \frac{ 4 \sqrt2 M^2 |J^*|^{\nicefrac12}\,\bar\zeta_n}{\kappa^{\rm RE}(|J^*|,1)}\bigwedge
		\frac{2\sqrt2 M \bar\zeta_n^{\nicefrac12}}{\kappa^{\rm RE}(|J^*|,1)^{\nicefrac12}}.
\end{align}
This inequality, in conjunction with the upper bound on $\bar\zeta_n$ used above,
completes the proof of the second claim.

\subsection{Proof of \Cref{prop:2}}
\label{ssec:proof:prop2}
\begin{figure}
\begin{center}
\includegraphics[width=0.8\textwidth]{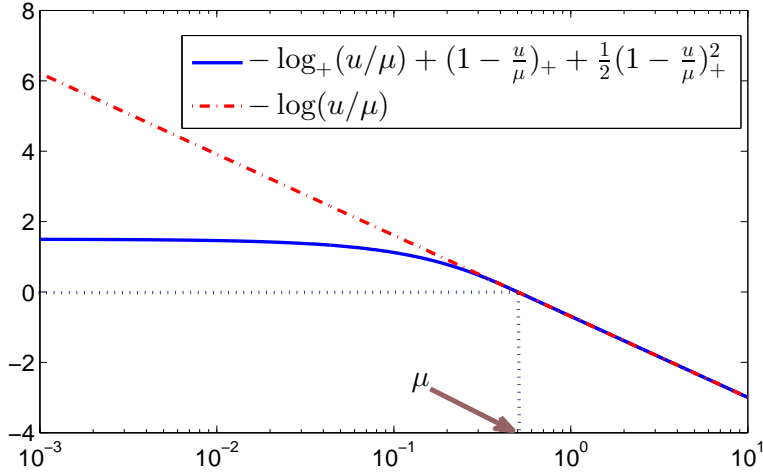}
\caption{The plot of the function $u\mapsto\bar\ell(u)$, used in the proof of
\Cref{prop:2}, superposed on the plot of the function $u\mapsto\ell(u)=-\log u$.
We see that the former is a strongly convex surrogate of the latter.}
\label{fig:ell}
\end{center}
\end{figure}
We repeat the proof of \Cref{maintheo1} with some small modifications. First of all, we replace the
function $\ell(u) = -\log(u)$ by the function
\begin{equation}\label{elBar}
\bar\ell(u) =
\begin{cases}
- \log(u/\mu), & \text{ if } u\ge \mu,\\
(1-\frac{u}{\mu})+\frac12(1-\frac{u}{\mu})^2, & \text{ if } u\in(0,\mu).
\end{cases}
\end{equation}
One easily checks that this function is twice continuously differentiable with a second
derivative satisfying $M^{-2}\le \bar\ell''(u)\le \mu^{-2}$ for every $u\in(0,M)$.
Furthermore, since $\bar\ell(u) = \ell(u/\mu)$ for every $u\ge \mu$, we have
$\bar L_n(\hat\bpi) = L_n(\hat\bpi)$, where we have used the notation $\bar L_n(\bpi)
= \frac1n\sum_{i=1}^n \bar\ell(f_{\bpi}(\bX_i))$. Therefore, similarly to \eqref{ineqOne},
we get
\begin{equation}\label{ineqOneb}
    \frac{1}{n}\sum_{i=1}^n \bar\ell(f_{\hat\bpi}(\bX_i)) \leq
    \frac{1}{n}\sum_{i=1}^n \bar\ell(f_{\bpi}(\bX_i)) -
    \frac{1}{2M^2n}\|\bar\bfZ(\hat\bpi-\bpi)\|^2_2,
\end{equation}
for every $\bpi\in\Pi^*(\mu)$. Let us define $\bar \varphi(\bpi,\bx) = \bar\ell(f_{\bpi}(\bx))-
\int \bar\ell(f_{\bpi})f^*d\nu$ and $\bar\Phi_n(\bpi) = \frac1n\sum_{i=1}^n
\bar\varphi(\bpi,\bX_i)$. We have
\begin{align}
\int \bar\ell(f_{\hat\bpi})\,f^*d\nu
    &\leq \int \bar\ell(f_{\bpi})\,f^*d\nu-\frac{1}{2M^2n}\|\bar\bfZ(\hat\bpi-\bpi)\|_2^2
        + \frac1n\sum_{i=1}^n \big(\varphi(\bpi,\bX_i)-\varphi(\hat\bpi,\bX_i)\big)\label{tempOraclIneqb}\\
    &\leq \int \bar\ell(f_{\bpi})\,f^*d\nu-\frac{1}{2M^2n}\|\bar\bfZ(\hat\bpi-\bpi)\|_2^2
        +\underbrace{\sup_{\bpi\in \Pi_n(0)}\|\nabla\bar\Phi_n(\bpi)\|_\infty}_{:=\xi_n} \|\hat\bpi-\bpi \|_1.
\end{align}
Notice that $\bpi\in\Pi^*(\mu)$ implies that $\bar\ell(f_{\bpi}) = \log\mu-\log f_{\bpi}$
and that $\bar\ell(f_{\hat\bpi}) \ge  \log \mu-\log f_{\hat\bpi} -(\log\mu-\log f_{\hat\bpi})_+ $.
Therefore, along the lines of the proof of \eqref{boundDevTwo} (see, namely, \eqref{ineqa}),
we get
\begin{align}
\label{tempOraclIneqC}
\KL(f^*||f_{\hat\bpi})
    &\leq \KL(f^*||f_{\bpi}) + \frac{2\xi_n^2M^2|J|}{\bar\kappa_{\hat\bSigma_n}(J,1)}
    +\int_{\calX}(\log\mu-\log f_{\hat\bpi})_+f^*d\nu.
\end{align}
We can repeat now the arguments of \Cref{boundEmpProcess} with some minor modifications. We first
rewrite $\xi_n$ as $\xi_n = \max_{l=1,\ldots,K} \xi_{l,n}$ with $\xi_{l,n} = \sup_{\bpi\in\Pi_n(0)}
|\partial_l \bar\Phi_n(\bpi)|$. One checks that the bounded difference inequality and the Efron-Stein
inequality can be applied with an additional factor 2, since for $F_l(\bfX) = \sup_{\bpi\in\Pi_n(0)}
|\partial_l \bar\Phi_n(\bpi)|$, we have
\begin{equation}
|F_l(\bfX)-F_l(\bfX')| \le \frac{2M}{n\mu} = \frac{2V}{n}.
\end{equation}
Therefore, for every $l\in[K]$, with probability larger than $1-(\delta/K)$, we have $\xi_{l,n}\le
\Ex[\xi_{l,n}]+ V(\frac{2\log(K/\delta)}{n})^{\nicefrac12}$ and $\var[\xi_n]\le (2V)^2/n$.
By the union bound, we obtain that with probability larger than $1-\delta$,
$\xi_{n}\le \max_l\Ex[\xi_{l,n}]+ V(\frac{2\log(K/\delta)}{n})^{\nicefrac12}$. Thus, to upper
bound $\Ex[\xi_{l,n}]$, we use the symmetrization argument:
\begin{align}
\Ex[\xi_{l,n}]
    & \le 2\Ex\bigg[\sup_{\bpi\in\Pi_n(0)} \bigg|\frac1n\sum_{i=1}^n \epsilon_i
            \bar\ell'(f_{\bpi}(\bX_i))f_l(\bX_i)\bigg|\bigg]\\
    &\le 2M\Ex\bigg[\sup_{\bpi\in\Pi_n(0)} \bigg|\frac1n\sum_{i=1}^n \epsilon_i
            \bar\ell'(f_{\bpi}(\bX_i))\bigg|\bigg]\qquad\text{\citep[Th.\ 11.5]{boucheron2013concentration}}\\
    &\le \frac{2M}{\mu}\Ex\bigg[\bigg|\frac1n\sum_{i=1}^n \epsilon_i\bigg|\bigg]+
        2M\Ex\bigg[\sup_{\bpi\in\Pi_n(0)} \bigg|\frac1n\sum_{i=1}^n \epsilon_i
            [\bar\ell'(f_{\bpi}(\bX_i))-\bar\ell'(0)]\bigg|\bigg].
\end{align}
Note that the function $\bar\ell'$, the derivative of $\bar\ell$ defined in \eqref{elBar},
is by construction Lipschitz with constant $1/\mu^2$. Therefore, in view of the contraction
principle,
\begin{align}
\Ex[\xi_{l,n}]
    &\le \frac{2M}{\mu}\Ex\bigg[\bigg(\frac1n\sum_{i=1}^n \epsilon_i\bigg)^2\bigg]^{\nicefrac12}+
        \frac{4M}{\mu^2}\Ex\bigg[\sup_{\bpi\in\Pi_n(0)} \frac1n\sum_{i=1}^n
        \epsilon_if_{\bpi}(\bX_i)\bigg]\\
    &\le \frac{2M}{\mu\sqrt{n}}+
        \frac{4M}{\mu^2}\Ex\bigg[\sup_{k\in[K]} \frac1n\sum_{i=1}^n
        \epsilon_i f_k(\bX_i)\bigg]\\
    &\le \frac{2M}{\mu\sqrt{n}}+
        \frac{8M^2}{\mu^2}\Big(\frac{\log K}{2n}\Big)^{\nicefrac12} \le \frac{2V^2(1+2\sqrt{2\log K})}{\sqrt{n}}.
\end{align}
As a consequence, we proved that with probability larger than $1-\delta$, we have
$\xi_n\le 8V^2(\frac{\log K}{n})^{\nicefrac12}$. This completes the proof of the first inequality.
In order to prove the second one, we simply change the way we have evaluated the term
$\int \bar\ell(f_{\hat\bpi})f^*$ in the left hand side of \eqref{tempOraclIneqb}. Since
$\bar\ell$ is strongly convex with a second order derivative bounded from below by $1/M^2$, we
have $\bar\ell(f_{\hat\bpi})\ge \bar\ell(f^*) + \bar\ell'(f^*)(f_{\hat\bpi}-f^*)+
\frac1{2M^2}(f_{\hat\bpi}-f^*)^2$. Since $f^*$ is always larger than $\mu$, the derivative
$\bar\ell'(f^*)$ equals $1/f^*$. Integrating over $\calX$, we get the second inequality of
the proposition.

\subsection{Auxiliary results}\label{ssec:auxiliary}

We start by a general convex result based on the strong convexity of the $-log$ function to
derive a bound on the estimated log-likelihood.
\begin{lemma}{}
\label{convexlemma}
Let us assume that $M =\max_{j\in[K]}\|f_j\|_\infty<\infty$. Then, for any $\bpi\in \BB^K_+$, it holds that
\begin{equation}
L_n(\hat\bpi) \le L_n(\bpi) -\frac1{2M^2n}\|\bar\bfZ(\hat\bpi-\bpi)\|_2^2.
\end{equation}
\end{lemma}

\begin{proof}
Recall that $\hat\bpi$ minimizes the function $L_n$ defined in \eqref{Phi} over $\Pi_n$. Furthermore,
the function $u\mapsto \ell(u)$ is clearly strongly convex with a second order derivative bounded from below
by $1/M^2$ over the set $u\in(0,M]$. Therefore, for every $\hat u\in(0,M]$, the function $\widetilde{\ell}$ given by:
\begin{equation}
    \widetilde{\ell}(u)=\ell(u)-\frac{1}{2M^2}(\hat{u}-u)^2,\quad u\in(0,M],
\end{equation}
is convex.
This implies that the mapping
\begin{equation}
\bpi\mapsto \widetilde L_n(\bpi) = L_n(\bpi)-\frac1{2 M^2 n}\|\bfZ(\hat\bpi-\bpi)\|_2^2
\end{equation}
is convex over the set $\bpi\in \BB^K_+$. This yields\footnote{We denote by $\partial g$ the sub-differential of a convex function $g$.}
\begin{equation}
\widetilde L_n(\bpi) -\widetilde L_n(\hat\bpi) \ge \sup_{\bv\in \partial\, \widetilde L_n(\hat\bpi)}\bv^\top (\bpi-\hat\bpi) ,\qquad \forall\bpi\in\BB^K_+.
\end{equation}
Using the Karush-Kuhn-Tucker conditions and the fact that $\hat\bpi$ minimizes $L_n$, we
get $\mathbf 0_K\in \partial\,L_n(\hat\bpi) = \partial\,\widetilde L_n(\hat\bpi)$. This
readily gives $\widetilde L_n(\bpi) -\widetilde L_n(\hat\bpi)\ge 0$, for any $\bpi\in\BB^K_+$.
The last step is to remark that $\bfZ(\hat\bpi-\bpi) = \bar\bfZ(\hat\bpi-\bpi)$, since both
$\hat\bpi$ and $\bpi$ have entries summing to one.
\end{proof}

The core of our results lies in the following proposition which bound the deviations of the empirical process part.

\begin{prop}[Supremum of Empirical Process]
\label{boundEmpProcess}
For any $\bpi\in\BB^K_+$ and $\bx\in\calX$, define $\varphi(\pi,\bx) = \int (\log f_{\bpi})f^* -\log f_{\bpi}(\bx)$ and consider
$\Phi_n(\bpi)=\frac1n\sum_{i=1}^n \varphi(\bpi,\bX_i)$. If $K\ge 2$, then for any $\delta \in (0,1)$, with probability at
least $1-\delta$, we have
\begin{equation}
  \zeta_n=\sup_{\bpi\in\Pi_n}\big\| \nabla \Phi_n(\bpi)\big\|_\infty \leq 8V^3\Big(\frac{\log(K/\delta)}{n}\Big)^{\nicefrac12}.
\end{equation}
Furthermore, we have $\Ex[\zeta_n]\le 4V^3\big(\frac{2\log(2K^2)}{n}\big)^{\nicefrac12}$ and $\var[\zeta_n]\le V^2/(2n)$.
\end{prop}
\begin{proof}
To ease notation, let us denote $g_{\bpi,l}(x) = \frac{f_l(x)}{f_{\bpi}(x)}-
\Ex\big[ \frac{f_l(\bX)}{f_{\bpi}(\bX)}\big]$ and
\begin{equation}
  F(\mathbf X) = \sup_{\bpi\in\Pi_n}\big\| \nabla \Phi_n(\bpi)\big\|_\infty
  =\sup_{(\bpi,l)\in\Pi_n\times[K]}\Big|  \frac{1}{n}\sum_{i=1}^n g_{\bpi,l}(\bX_i) \Big|,
\end{equation}
where $\bfX=(\bX_1,\dots,\bX_n)$.  To derive a bound on $F$, we will use the McDiarmid concentration
inequality that requires the bounded difference condition to hold for $F$. For some $i_0\in[n]$, let
$\bfX'=(\bX_1,\dots,\bX'_{i_0},\dots,\bX_n)$  be a new sample obtained from $\bfX$ by modifying the
$i_0$-th element $\bX_i$ and by leaving all the others unchanged. Then, we have
\begin{align}
F(\mathbf X)-F(\mathbf X')
        &= \sup_{(\bpi,l)\in\Pi_n\times[K]} \bigg|  \frac{1}{n}\sum_{i=1}^n g_{\bpi,l}\big(\bX_i\big) \bigg|-\sup_{(\bpi,l)\in\Pi\times[K]} \bigg|
            \frac{1}{n}\sum_{i=1}^n g_{\bpi,l}\big(\bX'_i\big) \bigg|\\
        &\leq \sup_{(\bpi,l)\in\Pi_n\times[K]}  \bigg|\frac{1}{n}\sum_{i=1}^n g_{\bpi,l}\big(\bX_i\big) - \frac{1}{n}\sum_{i=1}^n g_{\bpi,l}\big(\bX'_i\big)  \bigg|\\
        &=\sup_{(\bpi,l)\in\Pi_n\times[K]} \bigg| \frac{1}{n}\Big(g_{\bpi,l}\big(\bX_{i_0}\big) - g_{\bpi,l}\big(\bX'_{i_0}\big)\Big) \bigg|
        \leq \frac{V}{n},
\end{align}
where the last inequality is a direct consequence of assumption \eqref{densConst}. Therefore, using the McDiarmid concentration inequality recalled
in \Cref{McDiarmid} below, we check that the inequality
\begin{equation}
\label{concIneqZ}
  F(\mathbf X) \leq \Ex(F(\mathbf X))+V\sqrt{\frac{\log(1/\delta)}{2n}}
\end{equation}
holds with probability at least $1-\delta$. Furthermore, in view of the Efron-Stein
inequality, we have
\begin{equation}
\label{EfrStein}
  \var[\zeta_n] = \var[F(\mathbf X)] \leq \frac{V^2}{2n}.
\end{equation}

Let us denote $\calG:=\{(f_l/f_{\bpi})-1, (\bpi,l)\in\Pi_n\times [K]\}$ and $\mathfrak{R}_{n,q}(\calG)$ the Rademacher
complexity of $\calG$ given by
\begin{equation}
\label{radCompl}
  \mathfrak{R}_{n}(\calG) = \Ex_{\epsilon}\bigg[ \sup_{(\bpi,l)\in\Pi_n\times[K]}\bigg| \frac{1}{n}\sum_{i=1}^n
  \epsilon_i\Big(\frac{f_l(\bX_i)}{f_{\bpi}(\bX_i)}-1\Big)\bigg| \bigg],
\end{equation}
with $\epsilon_1,\dots,\epsilon_n$ independent and identically distributed Rademacher random variables independent
of $\bX_1,\dots,\bX_n$. Using the symmetrization inequality (see, for instance, Theorem~2.1 in~\cite{KoltBook2011})
we have
\begin{equation}
\label{rademacher}
  \Ex[F(\mathbf X)] = \Ex[\zeta_n] \leq 2\Ex[\mathfrak{R}_{n}(\calG)].
\end{equation}
\begin{lemma}
\label{boundRademComplex}
The Rademacher complexity defined in \eqref{radCompl} satisfies
\begin{equation}
\mathfrak{R}_n(\calG) \leq  4V^3\sqrt{\frac{\log K}{n}}.
\end{equation}
\end{lemma}
\begin{proof}
The proof relies on the contraction principle of \cite{LedouxTal:91} that we recall in
\Cref{sec:appendix_a} for the convenience. We apply this principle to the random variables
$X_{i,(\bpi,l)}=f_{\bpi}(\bX_i)/f_l(\bX_i)-1$ and to the function $\psi(x) = (1+x)^{-1}-1$.
Clearly $\psi$ is Lipschitz on
$[\frac{1}{V}-1,V-1]$ with the Lipschitz constant equal to $V^2$ and $\psi(0)=0$. Therefore
\begin{align}
   \mathfrak{R}_{n}(\calG)
			& \le  \Ex_{\epsilon}\bigg[ \sup_{(\bpi,l)} \frac{1}{n}\sum_{i=1}^n
                \epsilon_i\psi(\bX_{i,(\bpi,l)})\bigg]
                +\Ex_{\epsilon}\bigg[ \sup_{(\bpi,l)}\frac{1}{n}\sum_{i=1}^n
                \epsilon_i(-\psi)(\bX_{i,(\bpi,l)})\bigg]\\
			&\leq 2V^2\Ex_{\epsilon}\bigg[ \sup_{(\bpi,l)\in\Pi_n\times[K]}\frac{1}{n}
                \sum_{i=1}^n \epsilon_i \bX_{i,(\bpi,l)} \bigg]\\
			&= 2V^2 \Ex_{\epsilon}\bigg[ \sup_{(\bpi,l)\in\Pi_n\times[K]}\frac{1}{n}\sum_{i=1}^n \epsilon_i
					\bigg(\frac{f_{\bpi}(\bX_i)}{f_l(\bX_i)} -1\bigg) \bigg].
\end{align}
Expanding $f_{\bpi}(\bX_i)$ we obtain
\begin{align}
\Ex_{\epsilon}\bigg[ \sup_{(\bpi,l)}\frac{1}{n}\sum_{i=1}^n \epsilon_i \bigg(\frac{f_{\bpi}(\bX_i)}{f_l(\bX_i)}
				-1\bigg) \bigg]
		&= \Ex_{\epsilon}\bigg[ \sup_{(\bpi,l)} \sum_{k=1}^K \frac{\pi_k}{n}\sum_{i=1}^n \epsilon_i
				\bigg(\frac{f_{k}(\bX_i)}{f_l(\bX_i)} -1\bigg)  \bigg]\\
		&=\Ex_{\epsilon}\bigg[ \max_{k,l\in[K]} \frac{1}{n}\sum_{i=1}^n \epsilon_i
				\bigg(\frac{f_{k}(\bX_i)}{f_l(\bX_i)} -1\bigg) \bigg].
\end{align}
We apply now \Cref{Hoeffding1} with $s= (k,l)$, $N = K^2$, $a=-V$, $b= V$ and $Y_{i,s} = \epsilon_i
\big(\frac{f_{k}(\bX_i)}{f_l(\bX_i)} -1\big)$. This yields
\begin{align}
\Ex_{\epsilon}\bigg[ \max_{k,l\in[K]}  \frac{1}{n}\sum_{i=1}^n \epsilon_i
				\bigg(\frac{f_{k}(\bX_i)}{f_l(\bX_i)} -1\bigg) \bigg]
    \le 2V \Big(\frac{\log K^2}{2n}\Big)^{\nicefrac12}.
\end{align}
This completes the proof of the lemma.
\end{proof}
Combining inequalities (\ref{concIneqZ},\ref{rademacher}) and \Cref{boundRademComplex}, we get that the inequality
\begin{equation}
F(\bfX) \le 8V^3\Big(\frac{\log K}{n}\Big)^{\nicefrac12}+ V\Big(\frac{\log(1/\delta)}{2n}\Big)^{\nicefrac12}
\end{equation}
holds with probability at least $1-\delta$. Noticing that $V\ge 1$ and, for $K\ge 2$, $\delta\in(0,K^{-1/31})$
we have $8\sqrt{\log K}+ \sqrt{(\nicefrac12){\log(1/\delta)}}\le 8\sqrt{\log(K/\delta)}$,
we get the first claim of the proposition. The second claim is a direct consequence
of \Cref{boundRademComplex} and \eqref{rademacher}.
\end{proof}

\section{Proof of the lower bound for nearly-$D$-sparse aggregation}\label{sec:proof-lower}
We prove the minimax lower bound for estimation in Kullback-Leibler risk using
the following slightly adapted version of Theorem~2.5 from \cite{tsybakov2009Nonparametric}.
Throughout this section, we denote by $\lambda_{\min,\bSigma}(k)$ and $\lambda_{\max,\bSigma}(k)$,
respectively, the smallest and the largest eigenvalue of all $k\times k$ principal minors of
the matrix $\bSigma$.
\begin{theorem}
\label{tsy_main_theo_nonparam_est}
For some integer $L\ge 4$ assume that
$\calH_\calF(\gamma,D)$ contains $L$ elements $f_{\bpi^{(1)}},\dots,f_{\bpi^{(L)}}$
satisfying the following two conditions.
\vspace{-10pt}
\begin{enumerate}[(i)]
\item $\KL(f_{\bpi^{(j)}}||f_{\bpi^{(k)}})\geq2s>0$,  for all pairs $(j,k)$ such that $1\leq j<k\leq L$.
\item For product densities $f_{\ell}^n$ defined on $\calX^n$ by $f_{\ell}^n(\bx_1,\ldots,\bx_n) =
f_{\bpi^{(\ell)}}(\bx_1)\times\ldots\times f_{\bpi^{(\ell)}}(\bx_n)$ it holds
\begin{equation}
  \max_{\ell\in[L]} \KL(f_{\ell}^n||f_1^n) \leq \frac{\log L}{16}.
\end{equation}
\end{enumerate}
\vspace{-10pt}
Then
\begin{equation}
\inf_{\hat f}\sup_{f \in \calH_\calF(\gamma,D)}\Pb\!_{f}\big(\KL(f||\hat f)\geq s\big)\geq 0.17.
\end{equation}
\end{theorem}
To establish the bound claimed in \Cref{theorem:lower_bound}, we will split the problem into two parts,
corresponding to the following two subsets of $\calH_\calF(\gamma,D)$
\begin{equation}
\begin{array}{ll}
\calH_\calF(0,D) &= \big\{f_{\bpi} : \bpi\in\BB^K_+ \text{ such that }\exists \, J\subset [K]
\text{ with } \|\bpi_{J^c}\|_1 = 0 \text{ and } |J|\le D\big\}, \\
\calH_\calF(\gamma,1) &= \big\{f_{\bpi} : \bpi\in\BB^K_+ \text{ such that } \pi_1 = 1-\gamma \text{ and } \sum_{j=2}^K\pi_j = \gamma \big\}.
\end{array}
\end{equation}
We will show that over $\calH_\calF(0,D)$, we have a lower bound of order $\log(1+K/D)/n$ while over
$\calH_\calF(\gamma,1)$, a lower bound of order
$\big[\frac{\gamma^2}{n}\log\big(1+K/(\gamma\sqrt{n})\big)\big]^{\nicefrac12}$ holds true. Therefore,
the lower bound  over $\calH_\calF(\gamma,D)$ is larger than the average of these bounds.

For any $M\geq 1$ and $k\in [M-1]$, let  $\Omega_k^M$ be the subset of $\{0,1\}^M$ defined by
\begin{equation}
  \Omega_k^M := \Big \{\bomega \in \{0,1\}^M : \|\bomega\|_1=k \Big\}.
\end{equation}
Before starting, we remind here a version of the Varshamov-Gilbert lemma (see, for instance, \citep[Lemma 8.3]{RT11})
which will be helpful for deriving our lower bounds.
\begin{lemma}
\label{rigollet8.3}
Let $M\geq 4$ and $k\in[M/2]$ be two integers. Then there exist a subset $\Omega \subset \Omega_k^M$ and an absolute constant $C_1$
such that
\begin{equation}
  \|\bomega-\bomega'\|_1 \geq \frac{k+1}{4} \quad \forall \bomega,\bomega' \in \Omega \text{ \rm s.t. } \bomega \neq \bomega'
\end{equation}
and $L=|\Omega|$ satisfies $L\ge 4$ and
\begin{equation}
  \log L  \geq C_1 k\log\Big(1+\frac{eM}{ k}\Big).
\end{equation}
\end{lemma}
We will also use the following lemma that allows us to relate the KL-divergence $\KL(f_{\bpi}||f_{\bpi'})$
to the Euclidean distance between the weight vectors $\bpi$ and $\bpi'$.
\begin{lemma}
\label{lemma:KLboundsLT}
If the dictionary $\calF$ satisfies the boundedness assumption~\eqref{densConst}, then for any
$f_{\bpi}, f_{\bpi'}\in \calH_\calF(\gamma,D)$ we have
\begin{equation}
  \frac{1}{2V^2M}\,\|\bSigma^{\nicefrac12}(\bpi'-\bpi)\|_2^2
  \leq \KL(f_{\bpi}||f_{\bpi'})
  \leq \frac{V^2}{2m}\,\|\bSigma^{\nicefrac12}(\bpi'-\bpi)\|_2^2.
\end{equation}
\end{lemma}
\begin{proof}
Using the Taylor expansion, one can check that for any $u\in[1/L,L]$, we
have $(1-u)+ \frac{1}{2V^2}(u-1)^2 \leq -\log u \leq (1-u)+\frac{V^2}{2}(u-1)^2$.
Therefore,
\begin{equation}
  \frac{1}{2V^2}\int_{\calX}\Big(\frac{f_{\bpi'}}{f_{\bpi}}-1\Big)^2f_{\bpi}\,d\nu
  \leq \KL(f_{\bpi}||f_{\bpi'}) \leq \frac{V^2}{2}\int_{\calX}\Big(\frac{f_{\bpi'}}{f_{\bpi}}-1\Big)^2
  f_{\bpi}\,d\nu.
\end{equation}
Since $\calF$ satisfies the boundedness assumption, we get
\begin{equation}
\label{KLboundsLT}
  \frac{1}{2MV^2}\int_{\calX}\big(f_{\bpi'}-f_{\bpi}\big)^2d\nu
  \leq \KL(f_{\bpi}||f_{\bpi'}) \leq \frac{V^2}{2m}\int_{\calX}\big(f_{\bpi'}-f_{\bpi}\big)^2d\nu.
\end{equation}
The claim of the lemma follows from these inequalities and the fact that
$\int_{\calX}\big(f_{\bpi'}-f_{\bpi}\big)^2d\nu = \|\bSigma^{\nicefrac12}(\bpi'-\bpi)\|_2^2$.
\end{proof}

\subsection{Lower bound on $\calH_\calF(0,D)$} 

We show here that the lower bound ${(\nicefrac{D}n)\log(1+\nicefrac{e K}D)}
\wedge\big((\nicefrac1n){\log(1+\nicefrac{K}{\sqrt{n}})}\big)^{\nicefrac12}$ holds
when we consider the worst case  error for $f^*$ belonging to the set $\calH_\calF(0,D)$.
\begin{proposition}
\label{prop:lower:1}
If $\log(1+eK)\le  n$ then, for the constant
\begin{equation}
C_2=\frac{C_1 m\bar\kappa_{\bSigma}(2D,0)}{2^{9}V^2 M(C_1m\vee 4V^2\lambda_{\max,\bSigma}(2D))}
\ge \frac{C_1 m\varkappa_*}{2^{9}V^2 M(C_1m\vee 4V^2\varkappa^*)},
\end{equation}
we have
\begin{equation}
\inf_{\hat f}\sup_{f \in \calH_\calF(0,D)}\Pb\!_{f}\bigg(\KL(f||\hat f)\geq
C_2\,\frac{D\log(1+ K/D)}{n}\bigwedge\Big(\frac{\log(1+K/\sqrt{n})}{n}\Big)^{\nicefrac12}\bigg)
\geq 0.17.
\end{equation}
\end{proposition}

\begin{proof}
We assume that $D \leq \nicefrac{K}2$. The case $D> \nicefrac{K}2$ can be reduced to the
case $D = \nicefrac{K}2$ by using the inclusion $\calH_\calF(0,\nicefrac{K}2)\subset \calH_\calF(0,D)$. Let us set $A_1 = 4\vee {16 V^2\lambda_{\max,\bSigma}(2D)}/{(C_1 m)}$
and denote by $d$ the largest integer such that
\begin{equation}\label{eqFvSxOa}
d\le D\quad\text{and}\quad d^2\log\Big(1+\frac{e  K}{d}\Big)\le A_1 n.
\end{equation}
According to~\Cref{rigollet8.3}, there exists a subset
$\Omega = \{\bomega^{(\ell)}:\ell\in[L]\}$ of  $\Omega_{d}^{K}$ of cardinality $L\ge 4$
satisfying $\log L\geq {C_1 d}\log(1+{eK}/{d})$ such that for any pair of distinct
elements $\bomega^{(\ell)}$, $\bomega^{(\ell')}\in \Omega$ we have
$\|\bomega^{(\ell)}-\bomega^{(\ell')}\|_1 \geq d/4$. Using these binary vectors $\bomega^{(\ell)}$, we define the set $\mathcal{D} =
\{ \bpi^{(1)},\dots,\bpi^{(L)}\}\subset \BB_+^{K}$ as follows:
\begin{equation}\label{eqFvSxTw}
\bpi^{(1)}=\bomega^{(1)}/d,\quad
\bpi^{(\ell)}=(1-\varepsilon)\bpi^{(1)}+\varepsilon \bomega^{(\ell)}/d,\quad \ell=2,\ldots,L.
\end{equation}
Clearly, for every $\varepsilon\in[0,1]$, the vectors $\bpi^{(\ell)}$ belong to  $\BB^K_+$.
Furthermore, for any pair of distinct
values $\ell,\ell'\in[L]$, we have $\|\bpi^{(\ell)}-\bpi^{(\ell')}\|_q^q = (\varepsilon/d)^q
\|\bomega^{(\ell)}-\bomega^{(\ell')}\|_1\ge (\varepsilon/d)^qd/4$.  In view of~\Cref{lemma:KLboundsLT},
this yields
\begin{equation}\label{eqFvSxTh}
\KL(f_{\bpi^{(\ell)}}||f_{\bpi^{(\ell')}})
    \geq \frac{\bar\kappa_{\bSigma}(2d,0)}{4V^2M d} \big{\|}\bpi^{(\ell)} - \bpi^{(\ell')}\big{\|}_1^2
    \geq \frac{\bar\kappa_{\bSigma}(2D,0)}{64V^2M}\times\frac{\varepsilon^2}{d}.
\end{equation}
Let us choose
\begin{equation}
\varepsilon^2 = \frac{d^2 \log(1+e  K/d)}{n A_1}.
\end{equation}
It follows from \eqref{eqFvSxOa} that $\varepsilon\le 1$. Inserting this value of
$\varepsilon$ in \eqref{eqFvSxTh}, we get
\begin{equation}
\KL(f_{\bpi^{(\ell)}}||f_{\bpi^{(\ell')}}) \geq 2 C_2\,\frac{d\log(1+e  K/d)}{n}.
\end{equation}
This shows that condition (i) of \Cref{tsy_main_theo_nonparam_est} is satisfied with
$s= C_2\,(\nicefrac{d}{n})\log(1+e  K/d)$.  For the second condition of the same theorem,
we have
\begin{align}
\max_{\ell\in[L]}\KL(f^n_{\ell}||f^n_{1}) &= n\max_{\ell}\KL(f_{\bpi^{(\ell)}}||f_{\bpi^{(1)}})\\
&\leq \frac{ n V^2\lambda_{\max,\bSigma}(2d)}{2m}\max_\ell \|\bpi^{(\ell)} - \bpi^{(1)}\|_2^2\\
&\leq \frac{ n V^2\lambda_{\max,\bSigma}(2D)}{m} \times \frac{\varepsilon^2}{d},
\end{align}
since one can check that $\|\bpi^{(\ell)} - \bpi^{(1)}\|_2^2 \le (\varepsilon/d)^2
\|\bomega^{(\ell)}-\bomega^{(1)}\|_1\le 2\varepsilon^2/d$. Therefore, using the definition
of $\varepsilon$, we get
\begin{align}
\max_{\ell\in[L]}\KL(f^n_{\ell}||f^n_{1})
    &\le \frac{ n V^2\lambda_{\max,\bSigma}(2D)}{m}\times\frac{C_1dm\log(1+e  K/d)}
        {16 nV^2\lambda_{\max,\bSigma}(2D)}\\
    &= \frac{C_1d\log(1+e  K/d)}{16}\le \frac{\log L}{16}.
\end{align}
\Cref{tsy_main_theo_nonparam_est} implies that
\begin{equation}\label{eqFvSxTwa}
\inf_{\hat f}\sup_{f \in \calH_\calF(0,D)}\Pb\!_{f}\bigg(\KL(f||\hat f)\geq
C_2\,\frac{d\log(1+e  K/d)}{n}\bigg)\geq 0.17.
\end{equation}
We use the fact that $d$ is the largest integer satisfying
\eqref{eqFvSxOa}. Therefore, either $d+1>D$ or
\begin{equation}\label{eqFvSxTha}
(d+1)^2\log\Big(1+\frac{e  K}{d+1}\Big)\le A_1 n.
\end{equation}
If $d\ge D$, then the claim of the proposition follows from \eqref{eqFvSxTwa}, since
$d\log(1+eK/d)\ge D\log(1+eK/D)$. On the other hand, if \eqref{eqFvSxTha} is true, then
\begin{align}
  d\log(1+eK/d)
    &\ge \frac12 (d+1)\log(1+eK/(d+1))
    \ge \frac12 \big({A_1 n}{\log(1+eK/(d+1))}\big)^{\nicefrac12}.
\end{align}
In addition, $d^2\log(1+eK/d)\le A_1 n$ implies that $(d+1)^2 \le A_1 n$. Combining the last
two inequalities, we get the inequality $d\log(1+eK/d)  \ge \nicefrac12 \big({A_1 n}{\log(1+eK/\sqrt{A_1n})}\big)^{\nicefrac12}\ge  \big({n}{\log(1+eK/\sqrt{n})}\big)^{\nicefrac12}$.
Therefore, in view of \eqref{eqFvSxTwa}, we get the claim of the proposition.
\end{proof}

\subsection{Lower bound on $\calH_\calF(\gamma,1)$} 

Next result shows that the lower bound ${\frac{\gamma^2}n\log\big(1+\frac{K}{\gamma\sqrt{n}}\big)}$
holds for the worst case error when $f^*$ belongs to the set $\calH_\calF(\gamma,1)$.
\begin{proposition}\label{prop:lower:2}
Assume that
\begin{equation}
\Big(\frac{\log(1+eK)}{n}\Big)^{\nicefrac12}\le 2\gamma.
\end{equation}
Then, for the constant $C_3 =
\frac{C_1 m \bar\kappa_{\bSigma}(2D,0) }{2^{12}V^4 M\lambda_{\max,\bSigma}(2D)}$,
it holds that
\begin{equation}
\inf_{\hat f}\sup_{f \in \calH_\calF(\gamma,1)}\Pb\!_{f}\bigg(\KL(f||\hat f)\geq
C_3\,\Big\{{\frac{\gamma^2}n\log\Big(1+\frac{K}{\gamma\sqrt{n}}\Big)}\Big\}^{\nicefrac12}\bigg)\geq 0.17.
\end{equation}
\end{proposition}

\begin{proof}
Let $C>2$ be a constant the precise value of which will be specified later.
Denote by  $d$ the largest integer satisfying
\begin{equation}\label{eqFvSxFv}
	d\sqrt{\log(1+eK/d)} \leq C{\gamma\sqrt{n}}.
\end{equation}
Note that $d\ge 1$ in view of the condition $(\frac{\log(1+eK)}{n})^{\nicefrac12}\le 2\gamma$ of the
proposition. This readily implies that $d\leq C\gamma\sqrt{n}$ and, therefore,
\begin{equation}\label{eqFvSxSx}
\frac{\gamma}{d}
    \ge C^{-1}\Big\{{\frac1n\log\Big(1+\frac{eK}{C\gamma\sqrt{n}}\Big)}\Big\}^{\nicefrac12}
    \ge 2C^{-2}\Big\{{\frac1n\log\Big(1+\frac{K}{\gamma\sqrt{n}}\Big)}\Big\}^{\nicefrac12}.
\end{equation}
Let us first consider the case $d \leq (K-1)/2$. According to \Cref{rigollet8.3}, there
exists a subset $\Omega \subset \Omega_{d}^{K-1}$ of cardinality $L$ satisfying
$\log L\geq C_1\log\big(1+\frac{e(K-1)}{d}\big)$ and $\|\bomega^{(\ell)}-\bomega^{(\ell')}\|_1\geq d/4$
for any pair of distinct elements $\bomega,\bomega'$ taken from $\Omega$. With these
binary vectors in hand, we define the set $\calD \subset \BB_+^K$ of cardinality $L$ as follows:
\begin{equation}
\calD = \Big\{\bpi = \big(1-\gamma,{\gamma}\bomega/{d}\big) : \quad\omega \in \Omega\Big\}.
\end{equation}
It is clear that all the vectors of $\calD$ belong to $\calH_\calF(\gamma,1)$.
Let us fix now an element of $\calD$ and denote it by $\bpi^1$, the corresponding element of
$\Omega$ being denoted by $\bomega^1$. We have
\begin{align}
 \max_{\bpi\in\calD}\KL(f^n_{\bpi}||f^n_{\bpi^1})
        &\leq \frac{nV^2}{2m}\max_{\bpi\in\calD} \|\bSigma^{\nicefrac12}(\bpi-\bpi^{1})\|_2^2 \\
        &\leq \frac{nV^2\lambda_{\max,\bSigma}(2d)\gamma^2}{2m d^2} \max_{\bomega\in\Omega}
                \|\bomega-\bomega^{1}\|_2^2\\
        &\leq \frac{nV^2\lambda_{\max,\bSigma}(2d)\gamma^2}{m d}. \label{eqKLlowerBoundInq}
\end{align}
The definition of $d$ yields $(d+1)\sqrt{\log(1+eK/(d+1))}> C\gamma\sqrt{n}$, which implies that
\begin{align}
\frac{\gamma^2}{d}
        \leq 2(d+1)\frac{\gamma^2}{(d+1)^2}
        \leq 2(d+1)\frac{\log\big(1+eK/(d+1)\big)}{nC^2}
        \leq \frac{4d\log\big(1+e(K-1)/d\big)}{nC^2}.
\end{align}
Combined with \cref{eqKLlowerBoundInq}, this implies that
\begin{align}
\max_{\bpi\in\calD}\KL(f^n_{\bpi}||f^n_{\bpi^1})
        &\leq \frac{nV^2\lambda_{\max,\bSigma}(2d)}{m}\times \frac{4d\log\big(1+e(K-1)/d\big)}{nC^2}\\
        &= \frac{4V^2\lambda_{\max,\bSigma}(2d)}{m C^2}\times {d\log\big(1+e(K-1)/d\big)}.
\end{align}
Choosing
$$
C^2 = 2\vee\frac{64 V^2\lambda_{\max,\bSigma}(2d)}{C_1 m}
$$
we get that $\max_{\bpi\in\calD}\KL(f^n_{\bpi}||f^n_{\bpi^1})\le \frac1{16} C_1{d\log\big(1+e(K-1)/d\big)}\le
\frac{\log L}{16}$.

Furthermore, for any $\bpi,\bpi'\in\calD$, in view of~\Cref{lemma:KLboundsLT} and \eqref{eqFvSxSx}, we have
\begin{align}
\KL(f_{\bpi} || f_{\bpi'}) &\geq \frac{\bar\kappa_{\bSigma}(2d,0)}{4V^2M d} \big{\|}\bpi - \bpi'\big{\|}_1^2
=\frac{\bar\kappa_{\bSigma}(2d,0)\gamma^2}{4V^2M d^3} \|\bomega - \bomega'\big\|_1^2\\
&\geq \frac{\bar\kappa_{\bSigma}(2d,0)}{64V^2M}\times \frac{\gamma^2}{d}\\
&\geq \frac{\bar\kappa_{\bSigma}(2d,0)}{32V^2M C^2}\times
\Big\{{\frac{\gamma^2}n\log\Big(1+\frac{K}{\gamma\sqrt{n}}\Big)}\Big\}^{\nicefrac12}.
\end{align}
Since $\frac{\bar\kappa_{\bSigma}(2d,0)}{32V^2M C^2} = 2C_3$, this implies that
\Cref{tsy_main_theo_nonparam_est} can be applied, which leads to the inequality
\begin{equation}
\inf_{\hat f}\sup_{f \in \calH_\calF(\gamma,1)}\Pb\!_{f}\bigg(\KL(f||\hat f)\geq
C_3\,\Big\{{\frac{\gamma^2}n\log\Big(1+\frac{K}{\gamma\sqrt{n}}\Big)}\Big\}^{\nicefrac12}\bigg)\geq 0.17.
\end{equation}
To complete the proof of the proposition, we have to consider the case $d > (K-1)/2$. In this case,
we can repeat all the previous arguments for $d= K/2$ and get the desired inequality.
\end{proof}

\subsection{Lower bound holding for all densities} 

Now that we have lower bounds in probability for $\calH_\calF(0,D)$ and $\calH_\calF(\gamma,1)$, we can derive a lower bound in expectation for
$\calH_\calF(\gamma,D)$. In particular, to prove \Cref{theorem:lower_bound},
we will use the inequality
\begin{equation}\label{eqFvSxSev}
\calR\big(\calH_{\calF}(\gamma,D)\big)  \ge \inf_{\hat f}
\sup_{f^*\in \calH_\calF(0,D)\cup \calH_\calF(\gamma,1)} \Ex[\KL(f^*||\hat f)].
\end{equation}

\begin{proof}[Proof of \Cref{theorem:lower_bound}]
To ease notation, let us define
\begin{equation}
r(n,K,\gamma,D) = \bigg[\frac{\gamma^2}{n}
\log\bigg(1+\frac{K}{\gamma\sqrt{n}}\bigg)\bigg]^{\nicefrac12}
+ \frac{D\log(1+K/D)}{n}\bigwedge \Big(\frac{\log(1+K/\sqrt{n})}{n}\Big)^{\nicefrac12}.
\end{equation}
We first consider the case where the dominating term is the first one, that is
\begin{equation}\label{eqFvSxZO}
\bigg[\frac{\gamma^2}{n} \log\bigg(1+\frac{K}{\gamma\sqrt{n}}\bigg)\bigg]^{\nicefrac12}
\ge \frac{ 3 D\log(1+K/D)}{n}.
\end{equation}
On the one hand, since $D\ge 1$, we have
\begin{equation}\label{eqFvSxZT}
\frac{ 3D\log(1+K/D)}{n}\ge \frac{ \log(1+eK)}{n}.
\end{equation}
On the other hand, using the inequality $\log(1+x)\le x$, we get
\begin{align}
\bigg[\frac{\gamma^2}{n} \log\bigg(1+\frac{K}{\gamma\sqrt{n}}\bigg)\bigg]^{\nicefrac12}
    &\le \frac{\gamma}{\sqrt{n}}\bigg[\log(1+eK)+\log\bigg(1+\frac{1}{e^2\gamma^2n}\bigg)\bigg]^{\nicefrac12}\\
    &\le \gamma\bigg[\frac{\log(1+eK)}{n}\bigg]^{\nicefrac12}+\frac{\gamma}{\sqrt{n}}\bigg[\frac{1}{e^2\gamma^2n}\bigg]^{\nicefrac12}\\
    &\le \gamma\bigg[\frac{\log(1+eK)}{n}\bigg]^{\nicefrac12}+\frac{\log (1+eK)}{2n}.\label{eqFvSxZTh}
\end{align}
Combining \eqref{eqFvSxZO}, \eqref{eqFvSxZT} and \eqref{eqFvSxZTh}, we get
\begin{equation}
\Big(\frac{\log(1+eK)}{n}\Big)^{\nicefrac12}\le 2\gamma.
\end{equation}
This implies that we can apply \Cref{prop:lower:2}, which yields
\begin{equation}
\inf_{\hat f}\sup_{f \in \calH_\calF(\gamma,D)}\Pb\!_{f}\bigg(\KL(f||\hat f)\geq
C_3\,\Big\{{\frac{\gamma^2}n\log\Big(1+\frac{K}{\gamma\sqrt{n}}\Big)}\Big\}^{\nicefrac12}\bigg)\geq 0.17.
\end{equation}
In view of \eqref{eqFvSxZO}, this implies that
\begin{equation}\label{eqFvSxZFr}
\inf_{\hat f}\sup_{f \in \calH_\calF(\gamma,D)}\Pb\!_{f}\bigg(\KL(f||\hat f)\geq
\frac34C_3\,r(n,K,\gamma,D)\bigg)\geq 0.17.
\end{equation}
We now consider the second case, where the dominating term in the rate is the
second one, that is
\begin{equation}\label{eqFvSxZFv}
\bigg[\frac{\gamma^2}{n} \log\bigg(1+\frac{K}{\gamma\sqrt{n}}\bigg)\bigg]^{\nicefrac12}
\le \frac{ 3 D\log(1+K/D)}{n}\bigwedge
\Big(\frac{\log(1+K/\sqrt{n})}{n}\Big)^{\nicefrac12}.
\end{equation}
In view of \Cref{prop:lower:1}, we have
\begin{equation}
\inf_{\hat f}\sup_{f \in \calH_\calF(\gamma,D)}\Pb\!_{f}\bigg(\KL(f||\hat f)\geq
C_2\,\frac{ D\log(1+K/D)}{n}\bigwedge
\Big(\frac{\log(1+K/\sqrt{n})}{n}\Big)^{\nicefrac12}\bigg)\geq 0.17.
\end{equation}
In view of \eqref{eqFvSxZFv}, we get
\begin{equation}\label{eqFvSxZSx}
\inf_{\hat f}\sup_{f \in \calH_\calF(\gamma,D)}\Pb\!_{f}\bigg(\KL(f||\hat f)\geq
\frac14 C_2\,r(n,K,\gamma,D)\bigg)\geq 0.17.
\end{equation}
Thus, we have proved that $\log(1+eK)\le n$ implies that
$\inf_{\hat f}\sup_{f \in \calH_\calF(\gamma,D)}\Pb\!_{f}\big(\KL(f||\hat f)\geq
C_4\,r(n,K,\gamma,D)\big)\geq 0.17$ for some constant $C_4>0$, whatever the relation
between $\gamma$ and $D$. The desired lower bound follows now from the Tchebychev
inequality
$\Ex\big[\KL(f||\hat f)\big]\ge C_4\,r(n,K,\gamma,D)
\Pb\!_{f}\big(\KL(f||\hat f)\geq C_4\,r(n,K,\gamma,D)\big)$.
\end{proof}

\section*{Appendix A: Concentration inequalities} 
\label{sec:appendix_a}

This section contains some well-known results, which are recalled here for the sake of the self-containedness
of the paper.

\begin{theorem}\label{Hoeffding1}
For each $s=1,\ldots,N$, let $Y_{1,s},\ldots,Y_{n,s}$ be $n$ independent and zero mean random variables such that
for some real numbers $a,b$ we have $\Pb(Y_{i,s}\in[a,b])=1$ for all $i\in[n]$
and $s\in[N]$. Then, we have
\begin{equation}
  \Ex\Big[\max_{s\in[N]} \frac1n\sum_{i=1}^n Y_{i,s}\Big]\le (b-a)\Big(\frac{\log N}{2n}\Big)^{\nicefrac12},\quad
  \Ex\Big[\max_{s\in[N]}\Big| \frac1n\sum_{i=1}^n Y_{i,s}\Big|\Big]\le (b-a)\Big(\frac{\log(2N)}{2n}\Big)^{\nicefrac12}.
\end{equation}
\end{theorem}
\begin{proof}
We denote $Z_s = \frac1n\sum_{i=1}^n Y_{i,s}$ for $s=1,\ldots, N$ and $Z_s = -\frac1n\sum_{i=1}^n Y_{i,s}$
for $s=N+1,\ldots,2N$. For every $s\in[2N]$, the logarithmic moment generating function $\psi_s(\lambda) =
\log \Ex[e^{\lambda Z_{s}}]$ satisfies
\begin{equation}
\psi_s(\lambda) =  \log \big(\prod_i \Ex[e^{\lambda Y_{i,s}/n}])
    = \sum_{i=1}^n \log  \Ex[e^{\lambda Y_{i,s}/n}] \le \frac{\lambda^2(b-a)^2}{8n},
\end{equation}
where the last inequality is a consequence of the Hoeffding lemma (see, for instance,  Lemma~2.2
in~\citep{boucheron2013concentration}). This means that $Z_s$ is sub-Gaussian with variance-factor
$\nu = {(b-a)^2}/{4n}$. Therefore, Theorem~2.5 from \citep{boucheron2013concentration}
yields $\Ex[\max_s Z_s]\le \sqrt{2\nu\log(2N)}$, which completes the proof.
\end{proof}

We group and state together the bounded differences and the Efron-Stein inequalities (\cite{boucheron2013concentration},
Theorems~6.2 and 3.1, respectively).

\begin{theorem}\label{McDiarmid}
Assume that a function f satisfies the bounded difference condition: there exist constants $c_i$, $i=1,\ldots,n$
such that for all $i=1,\ldots, n$, all $X=(X_1,\dots,X_i,\dots,X_n)$ and $X'=(X_1,\dots,X'_i,\dots,X_n)$ where
only the $i^{th}$ vector is changed
\begin{equation}
  |f(X)-f(X')| \leq c_i.
\end{equation}
Denote
\begin{equation}
  \nu = \sum_{i=1}^n c_i^2.
\end{equation}
Let $Z=f(X_1,\dots,X_n)$ where $X_i$ are independent. Then, for every $\delta\in(0,1)$,
\begin{equation}
  \Pb\Big\{Z\le \Ex Z +\Big(\frac{\nu\log(1/\delta)}{2}\Big)^{\nicefrac12} \Big\} \geq 1-\delta,\qquad\text{and}\qquad \var[Z] \le \frac{\nu}{2}.
\end{equation}
\end{theorem}

Next we state the contraction principle of \citep{LedouxTal:91}; a proof can be found in
(\cite{boucheron2013concentration}, Theorem 11.6).
\begin{theorem}
Let $x_1,\dots,x_n$ be vectors whose real-valued components are indexed by $\bTau$,
that is, $x_i=(x_{i,s})_{s\in\bTau}$. For each $i=1,\dots,n$ let $\varphi_i:\RR\rightarrow\RR$
be a $1$-Lipschitz function such that $\varphi_i(0)=0$. Let $\epsilon_1,\dots,\epsilon_n$ be
independent Rademacher random variables, and let $\Psi:[0,\infty) \rightarrow \RR$ be a
non-decreasing convex function. Then
\begin{align}
  \Ex \bigg[\Psi\bigg(\frac{1}{2}\sup_{s\in\bTau}\bigg|\sum_{i=1}^n\epsilon_i\varphi_i(x_{i,s})\bigg| \bigg) \bigg]
  &\leq
  \Ex \bigg[\Psi\bigg(\sup_{s\in\bTau}\bigg|\sum_{i=1}^n\epsilon_i x_{i,s}\bigg| \bigg) \bigg]\\
  \Ex \bigg[\Psi\bigg(\sup_{s\in\bTau}\sum_{i=1}^n\epsilon_i\varphi_i(x_{i,s})\bigg) \bigg]
  &\leq  \Ex \bigg[\Psi\bigg(\sup_{s\in\bTau}\sum_{i=1}^n\epsilon_i x_{i,s} \bigg) \bigg] .
\end{align}
\end{theorem}

\section*{Acknowledgments}
The work of M.S.\ was partially supported by the French ``Agence Nationale de la Recherche'',
CIFRE n\textsuperscript{o} 2014/0517, and by {
ARTEFACT} (\url{www.artefact.is}).
The work of A.D.\ was partially supported by the grant Investissements d'Avenir
(ANR-11-IDEX-0003/Labex Ecodec/ANR-11-LABX-0047) and the chair ``LCL/GENES/Fon\-da\-tion du
risque, Nouveaux enjeux pour nouvelles donn\'ees''.

\bibliographystyle{plain}
\bibliography{ref}

\end{document}